\documentclass[12pt]{amsart}
\usepackage{latexsym,bm}
\usepackage{mathrsfs,amsmath,amssymb,amsthm,ascmac}
\usepackage[pdftex]{graphicx,color}

\theoremstyle{plain}
\newtheorem{theorem}{Theorem}[section]
\newtheorem{lemma}[theorem]{Lemma}
\newtheorem{prop}[theorem]{Proposition}
\newtheorem{fact}[theorem]{Fact}

\newtheorem{example}[theorem]{Example}
\newtheorem{cor}[theorem]{Corollary}
\newtheorem{obs}[theorem]{Observation}

\newtheorem{question}{Question}

\theoremstyle{definition}
\newtheorem{definition}[theorem]{Definition}
\newtheorem*{ack}{Acknowledgements}

\newtheorem*{claim}{Claim}

\DeclareMathOperator*{\diff}{\mathbb{D}}
\DeclareMathOperator*{\diffd}{\mathbb{D}^\ast}

\newcommand{\N}{\mathbb{N}}

\newcommand{\C}{\mathcal{C}}
\newcommand{\om}{\omega}
\newcommand{\ep}{\varepsilon}

\newcommand{\upto}{\upharpoonright}
\newcommand{\tto}{\rightrightarrows}

\newcommand{\omf}{\om^\rightarrow}
\newcommand{\omb}{\om^\leftarrow}

\newcommand{\fr}{{}^{\frown}}

\newcommand{\undef}{{\uparrow}}

\newcommand{\join}{\bigsqcup}

\newcommand{\ck}{\om_1^{\rm ck}}

\newcommand{\lnp}{\mbox{-}{\sf LNP}}
\newcommand{\lnpwo}{\mbox{-}{\sf LNP}_{\sf WO}}

\newcommand\tboldsymbol[1]{%
\protect\raisebox{0pt}[0pt][0pt]{%
$\underset{\widetilde{}}{\mathbf{#1}}$}\mbox{\hskip 1pt}}

\newcommand{\tpbf}[1]{\tboldsymbol{#1}}

\newcommand{\bfS}{\tpbf{\Sigma}}
\newcommand{\bfP}{\tpbf{\Pi}}
\newcommand{\pcolon}{\colon\!\!\!\subseteq}

\title{On some topics around the Wadge rank $\omega_2$}

\author[T. Kihara]{Takayuki Kihara}

\begin{document}

\begin{abstract}
Kechris and Martin showed that the Wadge rank of the $\om$-th level of the decreasing difference hierarchy of coanalytic sets is $\om_2$ under the axiom of determinacy.
In this article, we give an alternative proof of the Kechris-Martin theorem, by understanding the $\om$-th level of the decreasing difference hierarchy of coanalytic sets as the (relative) hyperarithmetical processes with finite mind-changes.
Based on this viewpiont, we also examine the gap between the increasing and decreasing difference hierarchies of coanalytic sets by relating them to the $\Pi^1_1$- and $\Sigma^1_1$-least number principles, respectively.
We also analyze Weihrauch degrees of related principles.
\end{abstract}

\maketitle

\section{Introduction}

\subsection{Summary}

In this article, we investigate topological complexity of sets in the difference hierarchy of coanalytic sets.
For a finite sequence $(A_m)_{m<n}$ of sets, its difference $\diff_{m\leq n}A_m$ is defined as follows:
\[\diff_{m\leq n}A_m=A_n\setminus(A_{n-1}\setminus(\dots\setminus (A_1\setminus A_0))).\]

One important aspect of the finite difference operator is that one can use this to represent exactly all finite Boolean combinations, and another is that it has a natural algorithmic interpretation, as we will see later.
There are two ways of extending the difference operator to certain transfinite sequences of sets.
The first operator $\diff$ is applicable to increasing sequences of sets, and the second one $\diffd$ is applicable to decreasing sequences; see Section \ref{preliminaries}.

For a class $\Gamma$ of sets, let us define $D_\xi(\Gamma)$ as the collection of all sets of the form $\diff_{\eta<\xi}A_\eta$ for some increasing sequence $(A_\eta)_{\eta<\xi}$ of $\Gamma$-sets, and $D^\ast_\xi(\Gamma)$ as the the collection of all sets of the form $\diff_{\eta<\xi}^\ast B_\eta$ for some decreasing sequence $(B_\eta)_{\eta<\xi}$ of $\Gamma$-sets.
Then, $(D_\xi(\Gamma))_{\xi<\om_1}$ is called the {\em increasing difference hierarchy} of $\Gamma$ sets, and $(D^\ast_\xi(\Gamma))_{\xi<\om_1}$ is called the {\em decreasing difference hierarchy} of $\Gamma$ sets.
In this article, we give a detailed analysis of these hierarchies for the case where $\Gamma$ is the collection of all coanalytic sets, i.e., $\Gamma=\tpbf{\Pi}^1_1$.

The formal definition (see Section \ref{preliminaries}) of the transfinite levels of the difference hierarchy is rather non-intuitive.
In order to make the meaning of the definition clearer, we describe a computational interpretation of the difference hierarchy, which is much easier to understand.
It is well-known that $\Pi^1_1$ is a higher analog of computable enumerability (based on a certain kind of ordinal-step computability; see e.g.~\cite{HinmanBook,SacksBook}).
As $\Delta^1_1$ is also known as hyperarithmetic, let us call a higher analog of computability by {\em hyp-computability} (so, one may refer to $\Delta^1_1$ as hyp-finite and $\Pi^1_1$ as hyp-c.e.)
Then, roughly speaking:
\begin{itemize}
\item[(1)]
The $\eta$-th level $D_\eta(\Pi^1_1)$ of the increasing difference hierarchy can be viewed as {\em hyp-computability with finite mind-changes along a countdown starting from $\eta$}.
\end{itemize}

More precisely, $A\in D_\eta(\Pi^1_1)$ if and only if there exists a hyp-computable learner guessing whether $n\in A$ or not through the following trial-and-error process: At first the ordinal $\eta$ is displayed in the countdown indicator, and the learner guesses $n\not\in A$, but during the process, the learner can change her mind and make another guess.
Each time the learner changes her mind, the learner has to choose some smaller ordinal than the current value displayed in the countdown indicator.
This newly chosen ordinal will be the next value displayed in the indicator.
As there is no infinite decreasing sequence of ordinals, this guarantees that the learner changes her mind at most finitely often.

This is a higher analog of ``computability with finite mind-changes along an ordinal countdown,'' which has been studied in various contexts, such as computational learning theory, see e.g.~\cite{FS93,AJS99}.
This notion must not be confused with {\em hyp-computability with ordinal mind-changes}, which corresponds to the decreasing difference hierarchy.
Indeed:
\begin{itemize}
\item[(2)]
The $\eta$-th level $D_\eta^\ast(\Pi^1_1)$ of the decreasing difference hierarchy can be thought of as {\em hyp-computability with at most $\eta$ mind-changes}.
\end{itemize}

To be more precise, as before, the hyp-computable learner guesses $n\not\in A$ at first, but during the process, the learner can change her mind and make another guess.
However, since this is an ordinal step computation, the learner has the opportunity to change her mind ordinal many times.
At a limit step, the learner may have changed her mind unboundedly, in which case her guess is reset to state ``$n\not\in A$'' (as in infinite time Turing computation \cite{Carl19}).
During the computation, the number of mind-changes must be kept below $\eta$.
However, if it reaches $\eta$, the learner has to terminate the process with the final guess ``$n\not\in A$''.

In particular, the ambiguous class $\Delta(D_\omega^\ast(\Pi^1_1))$ of the $\om$-th level of the decreasing difference hierarchy corresponds to {\em hyp-computability with finite mind-changes}, where $\Delta(\Gamma)=\Gamma\cap\neg\Gamma$.
Hence, ${\sf Diff}(\tpbf{\Pi}^1_1)\subseteq \Delta(D_\om^\ast(\tpbf{\Pi}^1_1))$, where ${\sf Diff}(\Gamma)$ stands for the whole increasing difference hierarchy $\bigcup_{\xi<\om_1}D_\xi(\Gamma)$.
Similarly, the whole decreasing difference hierarchy ${\sf Diff}^\ast(\Pi^1_1)=\bigcup_{\xi<\om_1}D_\xi^\ast(\Gamma)$ can be interpreted as {\em hyp-computability with a fixed countable ordinal mind-changes}.
A higher analog of the limit lemma (due to Monin; see \cite[Proposition 6.1]{BGM17}) also shows that hyp-computability with ordinal mind-changes corresponds to the sets which are $\Delta^0_1$ relative to sets in $\Pi^1_1\cup\Sigma^1_1$.
In summary, we get the following inclusions:
\[D_n(\tpbf{\Pi}^1_1)=D^\ast_n(\tpbf{\Pi}^1_1)\subsetneq \dots\subsetneq{\sf Diff}(\tpbf{\Pi}^1_1)\subseteq\Delta(D_\om^\ast(\tpbf{\Pi}^1_1))\subsetneq\dots\subsetneq{\sf Diff}^\ast(\tpbf{\Pi}^1_1)\subseteq\tpbf{\Delta}^0_1(\tpbf{\Pi}^1_1\cup\tpbf{\Sigma}^1_1),\]
where $\tpbf{\Delta}^0_1(\tpbf{\Pi}^1_1\cup\tpbf{\Sigma}^1_1)$ is the pointclass consisting of all sets which are $\tpbf{\Delta}^0_1$ relative to sets in $\tpbf{\Pi}^1_1\cup\tpbf{\Sigma}^1_1$; see Section \ref{sec:beyond-decreasing}.

So far, we have introduced two hierarchies of length $\om_1$; however, a question arises here:
Is it really the case that a hyp-computable procedure with finite mind-changes is always along some {\em countable} ordinal (i.e., some ordinal below $\om_1$) countdown?
Surprisingly, the answer is no.
On the one hand, Fournier \cite{FoPhD} showed that the Wadge rank of $D_{1+\alpha}(\tpbf{\Pi}^1_1)$ is $\phi_{\om_1}(\alpha)$ for $\alpha<\om_1$, where $\phi_{\om_1}$ is the $\om_1$-st Veblen function of base $\om_1$.
Hence, the Wadge rank of ${\sf Diff}(\tpbf{\Pi}^1_1)$ is $\phi_{\om_1}(\om_1)$.
On the other hand, according to Steel \cite{St81}, Kechris and Martin showed that the Wadge rank of $D_\om^\ast(\tpbf{\Pi}^1_1)$ is $\om_2$ under the axiom of determinacy.

\begin{theorem}[Kechris-Martin (unpublished); see Steel \cite{St81}]\label{thm:main-theorem}
Under the axiom of determinacy {\sf AD}, the order type of the Wadge degrees of $\Delta(D_\om^\ast(\tpbf{\Pi}^1_1))$ sets is $\om_2$.
\end{theorem}

This reveals the huge gap between ${\sf Diff}(\tpbf{\Pi}^1_1)$ and $D_\om^\ast(\tpbf{\Pi}^1_1)$.
In other words, a hyp-computable procedure with finite mind-changes is not necessarily along a countable ordinal countdown.

\begin{fact}[see also Fournier \cite{Fo16}]
Under {\sf AD}, the increasing difference hierarchy of coanalytic sets is strictly included in the $\om$th level of the decreasing difference hierarchy of coanalytic sets,  i.e., ${\sf Diff}(\tpbf{\Pi}^1_1)\subsetneq\Delta(D_\om^\ast(\tpbf{\Pi}^1_1))$.
\end{fact}

The proof for the lower bound $\om_2\leq{\rm otype}_W(\Delta(D_\om^\ast(\tpbf{\Pi}^1_1)))$ in Kechris-Martin's theorem has been written down in Steel \cite[Theorem 1.2]{St81} and Fournier \cite[Proposition 5.10]{FoPhD}.
For the upper bound, only a very rough idea, no more than two lines long, is commented on by Steel \cite{St81}.
According to Steel \cite{St81}, Martin's proof of the inequality ${\rm otype}_W(\Delta(D_\om^\ast(\tpbf{\Pi}^1_1)))\leq\om_2$ is based on the analysis of the ordinal games associated to Wadge games involving sets in $\Delta(D_\om^\ast(\tpbf{\Pi}^1_1))$.
In this article, we give a somewhat constructive alternative proof of Martin's upper bound which does not use any such techniques.


As a by-product of our constructive ideas, we can give a very clear solution to Fournier's problem,
which asks if the gap between the classes ${\sf Diff}(\tpbf{\Pi}^1_1)$ and $\Delta(D_\om^\ast(\tpbf{\Pi}^1_1))$ still exists even if we weaken the determinacy hypotheses (and may assume the axiom of choice).

\begin{question}[Fournier {\cite[Question 4.6]{Fo16}}]\label{four-main-question}
Is the equality between ${\sf Diff}(\tpbf{\Pi}^1_1)$ and $\Delta(D_\om^\ast(\tpbf{\Pi}^1_1))$ consistent under weaker determinacy hypothesis?
\end{question}

To solve Question \ref{four-main-question}, we give a natural set which belongs to the $\om$-th level of the decreasing difference hierarchy, but not to the increasing difference hierarchy (see Section \ref{sec:Fournier-problem}).

\begin{theorem}\label{thm:solution-to-Fournier}
Without any extra set-theoretic hypothesis, the increasing difference hierarchy of coanalytic sets is strictly included in the $\om$-th level of the decreasing difference hierarchy of coanalytic sets, i.e., ${\sf Diff}(\tpbf{\Pi}^1_1)\subsetneq\Delta(D_\om^\ast(\tpbf{\Pi}^1_1))$ holds, constructively.
\end{theorem}

Beyond the decreasing difference hierarchy, we also turn our attention to the inclusion ${\sf Diff}^\ast(\tpbf{\Pi}^1_1)\subseteq\tpbf{\Delta}^0_1(\tpbf{\Pi}^1_1\cup\tpbf{\Sigma}^1_1)$.
As mentioned above, the former corresponds to hyp-computability with fixed countable mind-changes, and the latter corresponds to hyp-computability with ordinal mind-changes by the relative higher limit lemma.
Then, it is natural to ask the following:

\begin{question}\label{question:beyond-difference}
Does the equality between ${\sf Diff}^\ast(\tpbf{\Pi}^1_1)$ and $\tpbf{\Delta}^0_1(\tpbf{\Pi}^1_1\cup\tpbf{\Sigma}^1_1)$ hold?
\end{question}

Our answer to Question \ref{question:beyond-difference} is that there is a huge gap between ${\sf Diff}^\ast(\tpbf{\Pi}^1_1)$ and $\tpbf{\Delta}^0_1(\tpbf{\Pi}^1_1\cup\tpbf{\Sigma}^1_1)$ (see Section \ref{sec:beyond-decreasing}), without assuming any extra set-theoretic hypothesis.

\begin{theorem}\label{thm:beyond-difference}
${\sf Diff}^\ast(\tpbf{\Pi}^1_1)\subsetneq\tpbf{\Delta}^0_1(\tpbf{\Pi}^1_1\cup\tpbf{\Sigma}^1_1)$.
\end{theorem}

\subsection{Preliminaries}\label{preliminaries}

For the basics of (effective) descriptive set theory, we refer the reader to Moschovakis \cite{mos07}.
For background and basic facts about Wadge degrees, see \cite{AnLo12}.
For higher computability, see e.g.~\cite{HinmanBook,SacksBook,BGM17}.

We use $\varphi_e^x$ to denote the $e$th partial computable function relative to an oracle $x$.
The least non-computable ordinal is denoted by $\om_1^{\rm ck}$.
Let ${\sf WO}\subseteq 2^{\om\times\om}\simeq 2^\om$ be the set of well-orders on $\N$.
For each $y\in{\sf WO}$, we also write $(\N,\leq_y)$ for the corresponding well-ordered set.
We use $|y|$ to denote the order type of $y$, and for each $a\in\N$, define $|a|_y$ as the order type of $\{b\in\N:b<_ya\}$.
It is known that ${\sf WO}$ is a $\tpbf{\Pi}^1_1$-complete set.
Indeed, if $P\subseteq\om^\om$ is $\Pi^1_1$, then there exists a computable function ${\bf o}_P$ such that, for any $x\in\om^\om$, $x\in P$ if and only if ${\bf o}_P(x)\in{\sf WO}$.
We often use this reduction to approximate a $\Pi^1_1$ set and a $\Pi^1_1$ function.
For instance, if $\psi\pcolon\om\to\om$ is a partial $\Pi^1_1$ function (i.e., the graph of $\psi$ is $\Pi^1_1$), then for any $s<\om_1^{\rm ck}$ the stage $s$ approximation of $\psi[s]$ can be defined as follows:
$\psi(n)[s]\downarrow=m$ if and only if the order type of ${\bf o}_P(n,m)$ is less than $s$, where $P$ is the graph of $\psi$.

For sets $A,B\subseteq\om^\om$, we say that {\em $A$ is Wadge reducible to $B$} ({\em written $A\leq_{\sf W}B$}) if there exists a continuous function $\theta\colon\om^\om\to\om^\om$, we have $A=\theta^{-1}[B]$.
A set $A$ is {\em selfdual} if $\neg A\leq_{\sf W}A$, where $\neg A$ is the complement of $A$.
For a pointclass $\Gamma$, we use $\neg\Gamma$ to denote its dual pointclass, that is, $\neg\Gamma=\{\neg A:A\in\Gamma\}$.
By Wadge's lemma \cite{Wadge83,AnLo12}, the Wadge degrees are semi-well-ordered under ${\sf AD}$, where ${\sf AD}$ stands for the axiom of determinacy.
Then, to each set $A\subseteq\om^\om$ one can assign the order type $|A|_{\sf W}$ of the collection of all nonselfdual sets $B\leq_{\sf W}A$, which is called the Wadge rank of $A$.

\section{Difference hierarchy}

\subsection{Difference of functions}

In this article, we deal with two {\em difference operators} $\diff$ and $\diff^\ast$.
However, the original definition of the increasing and decreasing difference operators is asymmetrical and rather hard to understand.
For the sake of clarity, we consider the difference operators for functions instead of sets, which yield a symmetric definition of the hierarchies.

Let $X$ and $Y$ be Polish spaces.
A sequence $(f_\xi)_{\xi<\eta}$ of partial functions $f_\xi\pcolon X\to Y$ is {\em dom-increasing} if $({\rm dom}(f_\xi))_{\xi<\eta}$ is increasing; and {\em dom-decreasing} if $({\rm dom}(f_\xi))_{\xi<\eta}$ is decreasing.
Fix $c\in Y\cup\{\undef\}$, where the symbol $\undef$ stands for ``undefined''.

\begin{definition}\label{def:diff-for-functions}
For a dom-increasing sequence $(f_\xi)_{\xi<\eta}$ of partial functions, we define $c\diff_{\xi<\eta}f_\xi\pcolon X\to Y$ as follows:
\[c\diff_{\xi<\eta}f_\xi(x)=
\begin{cases}
f_\gamma(x),&\mbox{ if $\gamma=\min\{\xi<\eta:x\in{\rm dom}(f_\xi)\}$},\\
c,&\mbox{ if no such $\gamma$ exists.}
\end{cases}
\]

For a dom-decreasing sequence $(f_\xi)_{\xi<\eta}$ of partial functions, we define $c\diff_{\xi<\eta}^\ast f_\xi\pcolon X\to Y$ as follows:
\[c\diffd_{\xi<\eta}f_\xi(x)=
\begin{cases}
f_\gamma(x),&\mbox{ if $\gamma=\max\{\xi<\eta:x\in{\rm dom}(f_\xi)\}$},\\
c,&\mbox{ if no such $\gamma$ exists.}
\end{cases}
\]
\end{definition}

Note that if $c\in Y$ then the resulting function is always total.
The usual increasing and difference hierarchies of $\tpbf{\Pi}^1_1$ sets are obtained by putting $c=0$ and considering constant functions $f_\eta\colon x\mapsto i$ with $\tpbf{\Pi}^1_1$ domains where $i\in\{0,1\}$; see Section \ref{sec:diff-for-sets}.

Hereafter, to simplify our argument, we assume $Y\subseteq\om$.
Let $cD_\eta(\tpbf{\Pi}^1_1)$ be the class of all functions of the form $c\diff_{\xi<\eta}f_\xi$ for a dom-increasing sequence $(f_\xi)_{\xi<\eta}$ of partial $\tpbf{\Pi^1_1}$ functions.
We also define $cD^\ast_\eta(\tpbf{\Pi}^1_1)$ in a similar manner.
To give a computability-theoretic interpretation of Definition \ref{def:diff-for-functions}, we also consider the lightface version of these classes.
For $\eta<\om_1^{\rm ck}$, let $cD_\eta({\Pi}^1_1)$ be the class of all functions of the form $c\diff_{\xi<\eta}f_\xi$ for a uniform $\Pi^1_1$ dom-increasing sequence $(f_\xi)_{\xi<\eta}$ of partial $\Pi^1_1$ functions.
We also define $cD^\ast_\eta({\Pi}^1_1)$ in a similar manner.
Here, a sequence $(f_\xi)_{\xi<\eta}$ is uniformly $\Pi^1_1$ if $\{(\xi,n,m):f_\xi(n)\downarrow=m\}$ is $\Pi^1_1$, where a computable ordinal $\xi$ is always identified with its notation; see also \ref{approx-mind-changes}.

\subsection{Approximation with mind-changes}\label{approx-mind-changes}

To explain the intuitive meaning of two difference hierarchies, we first introduce the notion of finite-change approximations.
For a detailed study of approximations with mind-changes in the context of higher computability theory, we refer the reader to Bienvenu-Greenberg-Monin \cite{BGM17}.
The results in Sections \ref{approx-mind-changes} and \ref{sec:diff-for-sets} are only used for us to get an intuition about two difference hierarchies, and will not be used in later sections.
For this reason, readers without prior knowledge of higher computability may skip Sections \ref{approx-mind-changes} and \ref{sec:diff-for-sets}.

Fix a $\Pi^1_1$ path $O_1$ through Kleene's $\mathcal{O}$ whose order type is $\ck$, and hereafter we identify $O_1$ with $\ck$.
For a function $\varphi\colon\om\times\om_1^{\rm ck}\to\om$, consider the set ${\sf mc}_\varphi(n)$ of all stages at which the value of $\varphi$ changes:
\[{\sf mc}_\varphi(n)=\{s<\ck:\varphi(n,s)\not=\varphi(n,s+1)\}\]

We say that $\varphi$ is a {\em finite-change function} if ${\sf mc}_\varphi(n)$ is a finite set for any $n\in\om$.
A function $\psi\colon\om\times\ck\to\eta$ is {\em antitone} if $s\leq t$ implies $\psi(n,s)\geq\psi(n,t)$ for any $n\in\om$.
An antitone function is a {\em countdown for }$\varphi\colon\om\times\om_1^{\rm ck}\to\om$ if for any $n\in\om$ and $s\in\ck$,
\[\varphi(n,s)\not=\varphi(n,s+1)\implies\psi(n,s)>\psi(n,s+1).\]

Observe that if $\varphi$ has a countdown, then $\varphi$ is a finite-change function.
If $\varphi$ changes at most finitely often, the limit value, $\lim_{s<t}\varphi(n,s)$, always exists, where
\[\lim_{s<t}\varphi(n,s)=m\iff \varphi(n,s)=m\mbox{ eventually holds for $s<t$}.\]

Here, we say that $A(s)$ eventually holds for $s<t$ if there exists $u<t$ such that $[u,t)\subseteq A$ holds, that is, for any $v$, $u\leq v<t$ implies $A(v)$.
We say that $\varphi$ is {\em continuous} if $\varphi(n,t)=\lim_{s<t}\varphi(n,s)$ for any limit ordinal $t<\ck$.

Let $\eta$ be a computable ordinal.
A function $\varphi\colon\om\times\ck\to\eta$ is {\em hyp-computable} if its graph is $\Pi^1_1$, where recall that $\ck$ is identified with the $\Pi^1_1$ set $O_1\subseteq\om$, and note that $\eta=\{s<\ck:s<\eta\}\subseteq O_1$.
Given $c\in\om\cup\{\undef\}$, we now show that $cD_\eta(\Pi^1_1)$ is equivalent to hyp-computability with finite mind-changes along $(\eta+1)$-countdown with the initial value $c$.


\begin{prop}\label{prop:diff-characterization1}
A function $f\pcolon\om\to\om$ belongs to $cD_\eta(\Pi^1_1)$ if and only if there exists a hyp-computable continuous function $\varphi\colon\om\times\ck\to\om$ such that for any $n\in\om$,
\begin{itemize}
\item $\varphi$ has an $(\eta+1)$-valued hyp-computable countdown,
\item $\varphi(n,0)=c$, and $f(n)=\lim_{s<\ck}\varphi(n,s)$.
\end{itemize}
\end{prop}

\begin{proof}
($\Rightarrow$)
Assume that $f=\diff_{\xi<\eta}f_\xi$ for a uniform sequence $(f_\xi)_{\xi<\eta}$ of partial $\Pi^1_1$ functions.
Fix a $\Delta^1_1$ approximation $(f_\xi[s])_{s<\ck}$ of $f_\xi$, so that $f_\xi[s]$ is a $\Delta^1_1$ function uniformly in $s<\ck$.
Then, define $\psi(n,s)=\min\{\xi<\eta:n\in{\rm dom}(f_\xi[s])\}$ if it exists; otherwise put $\psi(n,s)=\eta$.
It is clear that $\psi$ is a hyp-computable function, since given input $(n,s)$ we only need to simulate at most $\eta<\ck$ many hyp-algorithms for $\Delta^1_1$ functions $(f_\xi[s])_{\xi<\eta}$.
Then we define $\varphi(n,s)=f_{\psi(n,s)}(n)$ if $\psi(n,s)<\eta$; otherwise put $\varphi(n,s)=c$.
The function $\varphi$ is also hyp-computable.

Clearly, $\psi$ is an $(\eta+1)$-valued antitone function, which is a countdown for $\varphi$.
Let $\gamma<\eta$ be the least ordinal such that $n\in{\rm dom}(f_\gamma)$ if it exists.
Then $f(n)=f_\gamma(n)$ by the definition of the difference operator $\mathbb{D}$.
For such a $\gamma$, there exists $s_0<\ck$ such that $n\in{\rm dom}(f_\gamma[s_0])$, and for such an $s_0$, we have $\psi(n,s)=\gamma$ for any $s\geq s_0$ by minimality of $\gamma$.
Hence, $\varphi(n,s)=f_\gamma(n)=f(n)$ for any $s\geq s_0$.
This means that $\lim_{s<\ck}\varphi(n,s)=f(n)$.
If there is no such a $\gamma$, we have $\psi(n,s)=\eta$ by the definition of $\psi$, and therefore $\varphi(n,s)=c$ for any $s<\ck$.
Hence, $\lim_{s<\ck}\varphi(n,s)=c=f(n)$.

($\Leftarrow$)
Let $\varphi$ be a function in the assumption, and $\psi$ be a countdown for $\varphi$.
Given $\xi<\eta$ and $n\in\om$, if we see $\psi(n,s)\leq\xi$ for some $s<\ck$, then for the least such an $s$, define $f_\xi(n)=\varphi(n,s)$.
If there is no such an $s$, then $f_\xi(n)$ remains undefined.
Clearly, $(f_\xi)_{\xi<\eta}$ is dom-increasing.
Note that $(f_\xi)_{\xi<\eta}$ is a $\Pi^1_1$ sequence since $\varphi$ and $\psi$ are both hyp-computable.
We claim that $\diff_{\xi<\eta}f_\xi(n)=f(n)$, where $f(n)=\lim_{s<\ck}\varphi(n,s)$ by our assumption.
Let $\gamma<\eta$ be the least ordinal such that $n\in{\rm dom}(f_\gamma)$ if it exists.
Then $\diff_{\xi<\eta}f_\xi(n)=f_\gamma(n)$ by the definition of $\diff$.
By the definition of $f_\gamma$, the condition $n\in{\rm dom}(f_\gamma)$ implies that $\psi(n,s)\leq\gamma$ for some $s<\ck$, and by minimality of $\gamma$, there is no $s<\ck$ such that $\psi(n,s)<\gamma$.
Let $s_0<\ck$ be the least ordinal such that $\psi(n,s_0)=\gamma$.
Then we have $f_\gamma(n)=\varphi(n,s_0)$ by our definition of $f_\gamma$.
Since there is no $t>s_0$ such that $\psi(n,t)<\psi(n,s_0)=\gamma$, by the countdown condition, we have $\varphi(n,t)=\varphi(n,s_0)$ for any $t>s_0$.
This means that $\diff_{\xi<\eta}f_\xi(n)=f_\gamma(n)=\lim_{s<\ck}\varphi(n,s)$.
If there is no such a $\gamma$, $f_\xi(n)$ is undefined for all $\xi<\eta$, and thus, $\psi(n,s)=\eta$ for any $s<\ck$.
Since $\psi$ is a countdown for $\varphi$, we have $\varphi(n,s)=\varphi(n,0)=c$ for any $s<\ck$.
Therefore, $\diff_{\xi<\eta}f_\xi(n)=c=\lim_{s<\ck}\varphi(n,s)$.
\end{proof}

Next, let us move on to a function which may change infinitely often.
For such a function $\varphi$, in general, $\lim_{s<t}\varphi(n,s)$ does not necessarily exist.
Instead, for any constant $c\in\om$ and ordinal $\beta<\om_1^{\rm ck}$, we define 
\[
c\lim_{s<t}\varphi(n,s)=
\begin{cases}
m,&\mbox{ if }\varphi(n,s)=m\mbox{ eventually holds for $s<t$},\\
c,&\mbox{ if there exists no such $m$}.
\end{cases}
\]

We say that $\varphi$ is {\em $c$-semicontinuous} if $\varphi(n,t)=c\lim_{s<t}\varphi(n,s)$ for any limit ordinal $t<\ck$.
Note that any function $\varphi$ yields a semicontinuous function $\varphi^c$ by defining $\varphi(n,0)=c$; $\varphi^c(n,t+1)=\varphi(n,t)$ for any $t<\ck$; and $\varphi^c(n,t)=c\lim_{s<t}\varphi(n,s)$ for any limit ordinal $t<\ck$.
This is, for example, exactly the same as the behavior of infinite time Turing machines at limit steps.

Fix $c\in\om\cup\{\undef\}$, and let $\eta$ be a computable ordinal.
We characterize $cD_\eta^\ast(\Pi^1_1)$ as hyp-computability with at most $\eta$ mind-changes with the initial and reset value $c$.

\begin{prop}\label{prop:diff-characterization2}
A function $f\colon\om\to\om$ belongs to $cD_\eta^\ast(\Pi^1_1)$ if and only if there exists a hyp-computable $c$-semicontinuous function $\varphi\colon\om\times\ck\to\om$ such that for any $n\in\om$,
\begin{itemize}
\item ${\rm otype}({\tt mc}_\varphi(n))\leq\eta$,
\item $\varphi(n,0)=c$, and $f(n)=c\lim_{s<\ck}\varphi(n,s)$.
\end{itemize}
\end{prop}

\begin{proof}
($\Leftarrow$)
Let $\varphi$ be a function in the assumption, and for each $n\in\om$, let $(s^n_\xi)_{\xi<\lambda(n)}$ be the increasing enumeration of the set ${\tt mc}_\varphi(n)$ of all mind-change stages.
Since there is an order embedding of ${\tt mc}_\varphi(n)$ into $\eta$ by our assumption, we have $\lambda(n)\leq\eta$.
For any $n\in\om$ and $\xi<\lambda(n)$, define $f_\xi(n)=\varphi(n,s^n_\xi+1)$.
If $\xi\geq\lambda(n)$, $f_\xi(n)$ is undefined.
Clearly $(f_\xi)_{\xi<\eta}$ is dom-decreasing since we have ${\rm dom}(f_\xi)=\{n\in\om:\xi<\lambda(n)\}$.
Note also that  $(f_\xi)_{\xi<\eta}$ is a $\Pi^1_1$ sequence since $\varphi$ is hyp-computable, and ${\tt mc}_\varphi$ has a hyp-computable increasing enumeration.

We claim that $c\diff^\ast_{\xi<\eta}f_\xi(n)=f(n)$, where $f(n)=c\lim_{s<\ck}\varphi(n,s)$ by our assumption.
If $\lambda(n)$ is a successor ordinal, then $\gamma:=\lambda(n)-1$ is the greatest ordinal such that $n\in{\rm dom}(f_\gamma)$.
Then $c\diff^\ast_{\xi<\eta}f_\xi(n)=f_\gamma(n)$ by definition.
Then, $s_\gamma^n$ exists, and by maximality, there is no $t>s^n_\gamma$ such that $\varphi(n,t)\not=\varphi(n,t+1)$.
Hence, we have 
\[f(n)=c\lim_{s<\ck}\varphi(n,s)=\varphi(n,s^n_\gamma+1)=f_\gamma(n)=c\diffd_{\xi<\eta}f_\xi(n).\]

If $\lambda(n)$ is a limit ordinal, there is no greatest ordinal $\gamma$ such that $n\in{\rm dom}(f_\gamma)$, so $c\diff^\ast_{\xi<\eta}f_\xi(n)=c$.
Moreover, for $s^n=\sup_{\xi<\lambda(n)}s^n_\xi$, we have $\varphi(n,s^n)=c\lim_{s<s^n}\varphi(n,s)=c$ since $\varphi(n,s^n_\xi)\not=\varphi(n,s^n_{\xi}+1)$ for any $\xi<\lambda(n)$.
Therefore, $c\diff^\ast_{\xi<\eta}f_\xi(n)=c=c\lim_{s<s^n}\varphi(n,s)=f(n)$.

($\Rightarrow$)
Assume that $f=c\diff^\ast_{\xi<\eta}f_\xi$ for a dom-decreasing sequence $(f_\xi)_{\xi<\eta}$ of partial $\Pi^1_1$ functions.
Then, we have a hyp-approximation $(f_\xi[s])_{s<\ck}$ for $f$ for each $\xi<\eta$.
Fix $n\in\om$.
Let $s_\xi$ be the least stage $s$ such that $f_\xi(n)[s]$ is defined if such an $s$ exists.
Clearly, we may assume that $\zeta\leq\xi$ implies $s_\zeta\leq s_\xi$ since $(f_\xi)_{\xi<\eta}$ is dom-decreasing.
Moreover, we claim that if we choose a hyp-approximation for $f_\xi$ appropriately, we may assume that $s_\xi$ is successor for each $\xi<\eta$, and $(s_\xi)_{\xi<\eta}$ is strictly increasing.
To see this, put $s(\xi,t)=(\eta+1)\cdot t+\xi+1$.
Then, $s\colon\eta\times\ck\to\ck$ is injective.
Fix $\xi<\eta$, and first declare that $f'_\xi(n)[0]$ is undefined.
If $s=s(\xi,t)$ for some ordinal $t<\ck$, then put $f'_\xi(n)[s]=f_\xi(n)[t]$.
Assume that $s$ is not of the form $s(\xi,t)$.
If $s$ is successor, say $s=s'+1$, then put $f'_\xi(n)[s]=f'_\xi(n)[s']$.
If $s$ is limit, then put $f'_\xi(n)[s]=\lim_{t<s}f'_\xi(n)[t]$.
It is easy to see that $(f'_\xi[s])_{s<\ck}$ is a hyp-approximation for $f_\xi$ for each $\xi<\eta$.
Moreover, since $s(\xi,t)$ is successor, and $s$ is injective, one can see that this approximation has the desired property.
Then, replace $(f_\xi[s])_{s<\ck}$ with $(f'_\xi[s])_{s<\ck}$.

For a successor ordinal $s<\ck$, let $\psi(n,s)$ be the least ordinal $\xi<\eta$ such that $n\not\in{\rm dom}(f_\xi[s])$.
If there is no such $\xi$, put $\psi(n,s)=\eta$.
Note that $\psi(n,s)=\min(\{\xi<\eta:s<s_\xi\}\cup\{\eta\})$, so $\xi<\psi(n,s)$ if and only if $s_\xi\leq s$.
If $\psi(n,s)$ is successor, say $\psi(n,s)=\gamma+1$, then define $\varphi(n,s)=f_\gamma(n)$.
If $\psi(n,s)$ is limit, then define $\varphi(n,s)=c$.
For a limit ordinal $s<\ck$, define $\varphi(n,s)=c\lim_{t<s}\varphi(n,s)$.
Obviously, $\varphi$ is $c$-semicontinuous.
One can also check that $\varphi$ is hyp-computable.

We inductively define an order embedding $h\colon{\tt mc}_\varphi(n)\to\eta$ which, given $t\in{\tt mc}_\varphi(n)$, returns an ordinal less than $\psi(n,t+1)$.
Put $s=t+1$.
If $\psi(n,s)$ is successor, define $h(t)=\psi(n,s)-1$.
If $\psi(n,s)$ is limit, note that $s^\ast:=\sup\{s_\xi:\xi<\psi(n,s)\}<s$ since $(s_\xi)_{\xi<\eta}$ is strictly increasing and $s$ is successor.
Note that if $u$ is a successor ordinal with $s^\ast<u\leq s$ then $\psi(n,u)=\psi(n,s)$ by the definitions of $s^\ast$ and $\psi$.
Moreover, $t\in{\tt mc}_\varphi(n)$ implies that $\varphi(n,t)\not=\varphi(n,s)$, so we must have $s^\ast=t$.
First suppose that, for any $\xi<\psi(n,s)$ there exists $\zeta$ such that $\xi<\zeta<\psi(n,s)$ and $f_\xi(n)\not=f_\zeta(n)$.
In this case, as $(s_\xi)_{\xi<\eta}$ is a strictly increasing sequence of successor ordinals, we have $\varphi(n,s_\xi-1)=f_\xi(n)\not=f_\zeta(n)=\varphi(n,s_\zeta-1)$. 
This implies that $\varphi(n,t)=\varphi(n,s^\ast)=c\lim_{t<s^\ast}\varphi(n,t)=c$.
Moreover, $\varphi(n,s)=c$ since $\psi(n,s)$ is limit by our assumption.
This contradicts $t\in{\tt mc}_\varphi(n)$.

Thus, there exists $\xi<\psi(n,s)$ such that $\xi\leq\zeta<\psi(n,s)$ implies $f_\xi(n)=f_\zeta(n)$, for any $\zeta$.
Then, one might think that we can just define $h(t)$ as $\xi+1$; however recall that if $\psi(n,u)$ is a limit ordinal, then the value of $\varphi(n,u)$ is reset to $c$.
Thus, the value of $\varphi(n,u)$ may change even if $(f_\zeta(n))_{\xi\leq\zeta<\psi(n,s)}$ is constant.
Of course, if the value of $f_\xi(n)$ is $c$, there is no problem.
If $f_\xi(n)=c$, for any $u$ with $s_\xi\leq u\leq s$, we have $\varphi(n,u)=\varphi(n,s_\xi)=f_\xi(n)$.
In this case, we put $h(t)=\xi+1$, which implies $h(t)<\psi(n,s)$.
Note that $u<s$ and $u\in{\tt mc}_\varphi(n)$ implies $u<s_\xi$, so $u+1<s_{\xi+1}$ as $(s_\xi)_{\xi<\eta}$ is strictly increasing.
This implies $\psi(n,u+1)\leq\xi+1$ by the definition of $\psi$.
By the induction hypothesis, $h(u)<\psi(n,u+1)\leq\xi+1$.
Hence, $u<t$ implies $h(u)<h(t)$.

If $f_\xi(n)\not=c$, then there are two cases:
If $\psi(n,s)$ is a limit of limit ordinals, say $\psi(n,s)=\sup_{k\in\om}\lambda_k$ where each $\lambda_k$ is limit, then we have $\varphi(n,s_{\lambda_k}-1)=c$ since $\psi(n,s_{\lambda_k}-1)=\lambda_k$, which is limit.
Then $s^\ast=\sup\{s_{\lambda_k}:k\in\om\}$, and $t=s^\ast$ as seen before, so we have $\varphi(n,t)=c\lim_{u<s^\ast}\varphi(n,u)=c$.
However, $\varphi(n,u)=c$ as $\psi(n,s)$ is limit by our assumption.
Again, $t\in{\tt mc}_\varphi(n)$ implies that $\varphi(n,t)\not=\varphi(n,s)$, which is impossible.
Next, if $\psi(n,s)$ is not a limit of limit ordinals (while $\psi(n,s)$ is limit by our assumption), then $\psi(n,s)$ is of the form $\lambda+\om$.
Then choose $k\in\om$ such that $\xi\leq\lambda+k$, and define $h(t)=\lambda+k+1$, which implies $h(t)<\psi(n,s)$.
Note that for any successor ordinal $u$ with $s_{\lambda+k}\leq u<s^\ast=\sup\{s_{\lambda+\ell}:\ell\in\om\}$ we have $\psi(n,u)=\lambda+\ell$ for some $k<\ell<\om$.
In particular, $\psi(n,u)$ is successor, so $\varphi(n,u)=f_{\lambda+\ell}(n)=f_\xi(n)$.
Hence, for any $u$ with $s_{\lambda+k}\leq u\leq s$, we have $\varphi(n,u)=f_\xi(n)$.
Therefore, by the same argument as in the case $f_\xi(n)=c$, one can see that $u<t$ implies $h(u)<h(t)$.
Hence, $h$ is an order embedding.


We claim that $f(n)=c\lim_{s<\ck}\varphi(n,s)$.
Let us consider $\gamma=\max\{\xi<\eta:n\in{\rm dom}(f_\xi)\}$ if it exists.
Then, $f(n)=f_\gamma(n)$ since $f=\diff^\ast_{\xi<\eta}f_\xi$.
Let $s$ be the least stage such that $n\in{\rm dom}(f_\xi[s])$.
By maximality of $\gamma$, for any successor ordinal $t\geq s$ we have $\psi(n,t)=\gamma+1$, and thus $\varphi(n,t)=f_\gamma(n)$ by definition.
Therefore, we have $c\lim_{s<\ck}\varphi(n,s)=f_\gamma(n)=f(n)$.
If there is no such a $\gamma$, then $f(n)=c$.
Put $\lambda=\min\{\xi<\eta:n\not\in{\rm dom}(f_\xi)\}$
Then, $\lambda$ must be a limit ordinal as $\gamma$ is undefined.
Let us consider $(s_\xi)_{\xi<\lambda}$.
Note that $\xi\mapsto s_\xi\colon\lambda\to\ck$ is a total $\Pi^1_1$ function, and thus $\Delta^1_1$ since the domain is a computable ordinal.
Hence, by Spector's boundedness theorem (see e.g.~\cite[Corollary I.5.6]{SacksBook}), we have $s^\ast:=\sup\{s_\xi:\xi<\lambda\}<\ck$.
For any successor ordinal $s\geq s^\ast$, we have $\psi(n,s)=\lambda$, and thus $\varphi(n,s)=c$ since $\lambda$ is limit.
Hence, we have $c\lim_{s<\ck}\varphi(n,s)=c=f(n)$.
\end{proof}

As a corollary, one can see that for any $n\in\om$, and $\eta<\ck$, we have
\[cD_n(\Pi^1_1)=cD_n^\ast(\Pi^1_1)\subseteq\dots\subseteq cD_\om(\Pi^1_1)\subseteq\dots\subseteq cD_\eta(\Pi^1_1)\subseteq\dots\subseteq cD^\ast_\om(\Pi^1_1).\]

By relativizing Propositions \ref{prop:diff-characterization1} and \ref{prop:diff-characterization2}, one can show the similar results for Baire space $\om^\om$.

\subsection{Difference hierarchy for sets}\label{sec:diff-for-sets}

Now let us return to the original unintuitive definition of difference operators for sets.
For a countable ordinal $\xi$, if $(A_\eta)_{\eta<\xi}$ is an increasing sequence of sets, then its difference $\diff_{\eta<\xi}A_\eta$ is defined as follows:
\[
\diff_{\eta<\xi}A_\eta=
\bigcup_{\substack{\eta<\xi\\{\sf par}(\eta)\not={\sf par}(\xi)}}\left(A_\eta\setminus\bigcup_{\gamma<\eta}A_\gamma\right),
\]
where ${\sf par}(\xi)=1$ if $\xi$ is odd; otherwise, ${\sf par}(\xi)=0$.
If $n$ is a natural number, one can see that $\diff_{m\leq n}A_m=A_n\setminus(A_{n-1}\setminus(\dots\setminus (A_1\setminus A_0)))$.
If $(B_\eta)_{\eta<\xi}$ is a decreasing sequence of sets, then its difference $\diff^\ast_{\eta<\xi}B_\eta$ is defined as follows:
\[
\diffd_{\eta<\xi}B_\eta=\bigcup_{\substack{\eta<\xi\\ \eta\text{ even}}}\left(B_\eta\setminus B_{\eta+1}\right),
\]
where if $\xi$ is odd, put $B_\xi=\emptyset$.
If $n$ is a natural number, one can see that $\diff^\ast_{m\leq n}B_m=B_0\setminus(B_1\setminus(\dots\setminus (B_{n-1}\setminus B_n)))$.

Let $D_\eta(\tpbf{\Pi}^1_1)$ be the class of all sets of the form $\diff_{\xi<\eta}A_\xi$ for an increasing sequence $(A_\xi)_{\xi<\eta}$ of $\tpbf{\Pi^1_1}$ sets.
Similarly, let $D_\eta({\Pi}^1_1)$ be the class of all sets of the form $\diff_{\xi<\eta}f_\xi$ for a uniform $\Pi^1_1$ increasing sequence $(A_\xi)_{\xi<\eta}$ of $\Pi^1_1$ sets.
We also define the classes $D^\ast_\eta(\tpbf{\Pi}^1_1)$ and $D^\ast_\eta({\Pi}^1_1)$ in a similar manner.
To understand the relationship between the difference operators for sets and function, it is useful to introduce the following hybrid version of difference operators.
Let $X$ and $Y$ be Polish spaces, and fix $c\in Y\cup\{\undef\}$.

\begin{definition}\label{def:diff-for-functions}
For an increasing sequence $(A_\xi)_{\xi<\eta}$ of sets and a sequence $(f_\xi)_{\xi<\eta}$ of partial functions, we define $c\diff_{\xi<\eta}[f_\xi/A_\xi]\pcolon X\to Y$ as follows:
\[c\diff_{\xi<\eta}[f_\xi/A_\xi](x)=
\begin{cases}
f_\gamma(x),&\mbox{ if $\gamma=\min\{\xi<\eta:x\in A_\xi\}$},\\
c,&\mbox{ if no such $\gamma$ exists.}
\end{cases}
\]

For a decreasing sequence $(B_\xi)_{\xi<\eta}$ of sets and sequence $(f_\xi)_{\xi<\eta}$ of partial functions, we define $c\diff_{\xi<\eta}^\ast[f_\xi/B_\xi]\pcolon X\to Y$ as follows:
\[c\diffd_{\xi<\eta}[f_\xi/B_\xi](x)=
\begin{cases}
f_\gamma(x),&\mbox{ if $\gamma=\max\{\xi<\eta:x\in B_\xi\}$},\\
c,&\mbox{ if no such $\gamma$ exists.}
\end{cases}
\]
\end{definition}

Let $cD_\eta(\tpbf{\Sigma}^0_1/\tpbf{\Pi}^1_1)$ be the class of all sets of the form $c\diff_{\xi<\eta}[f_\xi/A_\xi]$ for an increasing sequence $(A_\xi)_{\xi<\eta}$ of $\tpbf{\Pi^1_1}$ sets, and a sequence $(f_\xi)_{\xi<\eta}$ of continuous functions.
Similarly, let $cD_\eta(\Sigma^0_1/\Pi^1_1)$ be the class of all sets of the form $c\diff_{\xi<\eta}[f_\xi/A_\xi]$ for a uniform $\Pi^1_1$ increasing sequence $(A_\xi)_{\xi<\eta}$ of $\Pi^1_1$ sets, and a computable sequence $(f_\xi)_{\xi<\eta}$ of computable functions.
We also define the classes $cD^\ast_\eta(\tpbf{\Sigma}^0_1/\tpbf{\Pi}^1_1)$ and $cD^\ast_\eta(\Sigma^0_1/{\Pi}^1_1)$ in a similar manner.
Obviously, $cD_\eta(\tpbf{\Sigma}^0_1/\tpbf{\Pi}^1_1)\subseteq cD_\eta(\tpbf{\Pi}^1_1)$ and $cD^\ast_\eta(\tpbf{\Sigma}^0_1/\tpbf{\Pi}^1_1)\subseteq cD^\ast_\eta(\tpbf{\Pi}^1_1)$.
The lightface versions also hold.

These hybrid difference operators seem relevant for studying $\sigma$-continuous functions ($\om$-decomposable functions; see e.g.~\cite{GKN21}).
As in Propositions \ref{prop:diff-characterization1} and \ref{prop:diff-characterization2}, the classes $cD_\eta(\Sigma^0_1/\Pi^1_1)$ and $cD^\ast_\eta(\Sigma^0_1/{\Pi}^1_1)$ are characterized as hyp-computability of an index $\gamma$ with mind-changes.
Such an index-guessing has been extensively studied in the theory of inductive inference (identification in the limit; see \cite{JOSW}).

Observe that the characteristic function of a set in $D_\eta(\tpbf{\Pi}^1_1)$ belongs to $0D_\eta(\tpbf{\Sigma}^0_1/\tpbf{\Pi}^1_1)$:
Given an increasing sequence $(A_\xi)_{\xi<\eta}$ of sets, define $f_\xi\colon A_\xi\to 2$ by $f_\xi(x)=1$ if ${\sf par}(\xi)\not={\sf par}(\eta)$; otherwise $f_\xi(x)=0$.
Similarly, the characteristic function of a set in $D^\ast_\eta(\tpbf{\Pi}^1_1)$ belongs to $0D_\eta^\ast(\tpbf{\Sigma}^0_1/\tpbf{\Pi}^1_1)$:
Given a decreasing sequence $(A_\xi)_{\xi<\eta}$ of sets, define $f_\xi\colon A_\xi\to 2$ by $f_\xi(x)=1$ if ${\sf par}(\xi)=0$; otherwise $f_\xi(x)=0$.

As a corollary of Proposition \ref{prop:diff-characterization1}, the class $D_\eta(\Pi^1_1)$ is characterized as hyp-computability with finite mind-changes along $(\eta+1)$-countdown with the initial value $0$.

\begin{cor}
A set $A\subseteq\om$ belongs to $D_\eta(\Pi^1_1)$ if and only if there exists a hyp-computable continuous finite-change function $\varphi\colon\om\times\ck\to 2$ such that for any $n\in\om$,
\begin{itemize}
\item $\varphi$ has an $(\eta+1)$-valued hyp-computable countdown,
\item $\varphi(n,0)=0$, and $A(n)=\lim_{s<\ck}\varphi(n,s)$.
\end{itemize}
\end{cor}

Similarly, as a corollary of Proposition \ref{prop:diff-characterization2}, the class $D^\ast_\eta(\Pi^1_1)$ is characterized as hyp-computability with at most $\eta$ mind-changes with the initial value $0$.

\begin{cor}
A set $A\subseteq\om$ belongs to $D^\ast_\eta(\Pi^1_1)$ if and only if there exists a hyp-computable $c$-semicontinuous function $\varphi\colon\om\times\ck\to\om$ such that for any $n\in\om$,
\begin{itemize}
\item ${\rm otype}({\tt mc}_\varphi(n))\leq\eta$,
\item $\varphi(n,0)=0$, and $A(n)=0\lim_{s<\ck}\varphi(n,s)$.
\end{itemize}
\end{cor}

It is easy to show the following analogues of Post's theorem.

\begin{prop}
A set $A\subseteq\om$ belongs to $\Delta(D_\eta(\Pi^1_1))$ if and only if there exists a hyp-computable continuous finite-change function $\varphi\colon\om\times\ck\to 2$ such that for any $n\in\om$,
\begin{itemize}
\item $\varphi$ has an $\eta$-valued hyp-computable countdown,
\item and $A(n)=\lim_{s<\ck}\varphi(n,s)$.
\end{itemize}
\end{prop}

\begin{prop}
A set $A\subseteq\om$ belongs to $\Delta(D^\ast_\eta(\Pi^1_1))$ if and only if there exists a hyp-computable $c$-semicontinuous function $\varphi\colon\om\times\ck\to\om$ such that for any $n\in\om$,
\begin{itemize}
\item ${\rm otype}({\tt mc}_\varphi(n))<\eta$,
\item and $A(n)=c\lim_{s<\ck}\varphi(n,s)$.
\end{itemize}
\end{prop}

In particular, $\Delta(D^\ast_\om(\Pi^1_1))$ corresponds to hyp-computability with finite mind-changes.

\section{Solution to Fournier's problem}\label{sec:3}

\subsection{Weihrauch lattice}\label{sec:Weihrauch}

Let us explain that the class $cD_\eta(\tpbf{\Sigma}^0_1/\tpbf{\Pi}^1_1)$ is to some extent a natural one in terms of the Weihrauch lattice.
This perspective will also be used to solve Fournier's Question \ref{four-main-question}.
The study of the Weihrauch lattice aims to measure the computability theoretic difficulty of finding a choice function witnessing the truth of a given $\forall\exists$-theorem (cf.\ \cite{pauly-handbook}) as an analogue of reverse mathematics \cite{SOSOA:Simpson}.
The notion of Weihrauch degree is used as a tool to classify certain $\forall\exists$-statements by identifying each $\forall\exists$-statement with a partial multivalued function.
Informally speaking, a (possibly false) statement $S\equiv\forall x\in X\;[Q(x)\rightarrow\exists yP(x,y)]$ is transformed into a partial multivalued function $f\pcolon X\rightrightarrows Y$ such that ${\rm dom}(f)=\{x:Q(x)\}$ and $f(x)=\{y:P(x,y)\}$.
Then, measuring the degree of difficulty of witnessing the truth of $S$ is identified with that of finding a choice function for $f$.
Here, we consider choice problems for partial multivalued functions rather than relations in order to distinguish the hardest instance $f(x)=\emptyset$ and the easiest instance $x\in X\setminus{\rm dom}(f)$.

If one only considers subspaces of $\N^\N$, one can use the following version of Weihrauch reducibility:
For partial multivalued functions $f$ and $g$, we say that $f$ is {\em Weihrauch reducible to $g$} (written $f\leq_{\sf W}g$) if there are partial computable functions $h$ and $k$ such that the following holds:
Given an instance $x$ of $f$-problem (i.e., $x\in{\rm dom}(f)$), if we know a solution $y$ to the instance $h(x)$ of $g$-problem (i.e., $y\in g(h(x))$), then the algorithm $k$ tells us that $k(x,y)$ is a solution to the instance $x$ of $f$-problem (i.e., $k(x,y)\in f(x)$).
In other words,
\[(\forall x\in{\rm dom}(f))(\forall y)\;[y\in g(h(x))\implies k(x,y)\in f(x)].\]

The functions $h$ and $k$ are often called an {\em inner reduction} and an {\em outer reduction}, respectively.
To discuss Weihrauch reducibility in other spaces, we introduce some auxiliary concepts.
A {\em representation} of a set $X$ is a partial surjection $\delta_X\pcolon\om^\om\to X$.
If $\delta_X(p)=x$, then $p$ is called a {\em $\delta_X$-name of $x$}\index{name} (or simply, a {\em name of $x$} if $\delta_X$ is clear from the context).
A pair of a set and its representation is called a {\em represented space}.

\begin{example}
Perhaps, one of the best known examples of represented spaces in descriptive set theory is the space ${\sf Bor}$ of Borel sets in a Polish space, where consider the representation $\delta_{\sf Bor}\pcolon\om^\om\to{\sf Bor}$ defined by $\delta_{\sf Bor}(p)=A$ if and only if $p$ is a Borel code of $A$.
In other words, a $\delta_{\sf Bor}$-name of $A$ is exactly a Borel code of $A$.
\end{example}

\begin{definition}[see also \cite{pauly-handbook}]
Let $X$, $Y$, $Z$ and $W$ be represented spaces with representations $\delta_X$, $\delta_Y$, $\delta_Z$ and $\delta_W$, respectively.
For partial multivalued functions $f\pcolon X\tto Y$ and $g\pcolon Z\tto Y$, we say that $f$ is {\em Weihrauch reducible to $g$} (written $f\leq_{\sf W}g$) if there are partial computable functions $h$ and $k$ such that the following holds:
Given a $\delta_X$-name ${\tt x}$ of an instance $x$ of $f$-problem, the algorithm $h$ tells us a $\delta_Z$-name $h({\tt x})$ of an instance $x^\ast$ of $g$-problem, and if we know a $\delta_W$-name ${\tt y}$ of a solution $y$ to the instance $x^\ast$ of $g$-problem, then the algorithm $k$ tells us that $k({\tt x},{\tt y})$ is a $\delta_Y$-name of a solution to the instance $x$ of $f$-problem.
In other words,
\[(\forall {\tt x}\in{\rm dom}(f\circ\delta_X))(\forall {\tt y})\;[\delta_W({\tt y})\in g\circ\delta_Z(h({\tt x}))\implies \delta_Y\circ k({\tt x},{\tt y})\in f\circ\delta_X({\tt x})].\]
\end{definition}


We now consider the following $\forall\exists$-principles related to the difference hierarchy:
\begin{itemize}
\item {\em $\Gamma$-least number principle}: For any nonempty $\tpbf{\Gamma}$ set $A\subseteq\om$, there exists the least element of $A$.
\item {\em $\Gamma$-counting}: For any finite $\tpbf{\Gamma}$ set $A\subseteq\om$, the value $\# A$ exists.
\end{itemize}

We consider the case where $\Gamma$ is either ${\Pi}^1_1$ or ${\Sigma}^1_1$.
For such a $\Gamma$, note that if $X$ is a Polish space then the collection $\tpbf{\Gamma}(X)$ of all $\tpbf{\Gamma}$ subsets of $X$ has a total representation $\delta_\Gamma\colon\om^\om\to\tpbf{\Gamma}(X)$.
For instance, if $X=\N$ and $\Gamma={\Pi}^1_1$ then, for any $e\in\om$ and $p\in\om^\om$, the concatenation $e\fr p$ is a $\delta_{\Pi^1_1}$-name of $A\subseteq\om$ if and only if $A$ is the $e$-th $\Pi^1_1(p)$ set.
Hereafter, we also use $P_x$ to denote $\delta_{\Pi^1_1}(x)$; that is, $P_x$ is the $\tpbf{\Pi}^1_1$ set coded by $x$.

\begin{definition}
We define the $\Gamma$-least number principle $\Gamma\lnp\pcolon\tpbf{\Gamma}(\om)\to\om$ as follows:
\[
\Gamma\lnp(A)=
\begin{cases}
\min A,&\mbox{ if }A\not=\emptyset,\\
\mbox{undefined},&\mbox{ if }A=\emptyset.
\end{cases}
\]

We define the $\Gamma$-counting principle $\#\Gamma\pcolon\tpbf{\Gamma}(\om)\to\om$ as follows:
\[
\#\Gamma(A)=
\begin{cases}
\# A,&\mbox{ if }\#A\mbox{ is finite},\\
\mbox{undefined},&\mbox{ otherwise}.
\end{cases}
\]
\end{definition}

Let $(X,\delta_X)$ be a represented space.
We say that a partial function $f\pcolon X\to\om$ is $cD_\om(\Gamma)$-complete if $f\circ\delta_X\pcolon \om^\om\to\om$ belongs to $cD_\om(\Gamma)$, and any $cD_\om(\Gamma)$-function $g\pcolon\om^\om\to\om$ is Weihrauch reducible to $f$.
We define $cD^\ast_\om(\Gamma)$-completeness in a similar manner.
We now consider the case $c={}\uparrow$ (indicating ``undefined'').

\begin{prop}
$\Pi^1_1\lnp$ is $\undef D_\om(\Sigma^0_1/\Pi^1_1)$-complete.
\end{prop}

\begin{proof}
To see that $\Pi^1_1\lnp$ is in $\undef D_\om(\Sigma^0_1/\Pi^1_1)$, define $A_n=\{x:(\exists k\leq n)\;k\in P_x\}$ where $P_x$ is the $x$th $\tpbf{\Pi}^1_1$ set, and consider the constant function $c_n\colon x\mapsto n$.
Then, $(A_n)_{n<\om}$ is an increasing sequence of $\tpbf{\Pi}^1_1$ sets.
One can easily see that $\min P_x=\min\{n\in\om:x\in A_n\}$ whenever $P_x$ is empty.
Recall from Definition \ref{def:diff-for-functions} that $\undef\diff_{n<\om}[c_n/A_n](x)=\min\{n\in\om:x\in A_n\}$ if it exists.
Therefore, $\undef\diff_{n<\om}[c_n/A_n]$ is a realizer for $\Pi^1_1\lnp$.
To show the completeness, assume that a sequence $(A_n,f_n)_{n<\om}$ of pairs of $\tpbf{\Pi}^1_1$ sets and continuous functions is given.
To see that $\undef\diff_{n<\om}[f_n/A_n]$ is Weihrauch reducible to $\Pi^1_1\lnp$, let us consider the inner reduction $h$ which maps $x$ to a $\delta_{\Pi^1_1}$-name of $Q_x=\{n\in\om:x\in A_n\}$, and the outer reduction $k$ which maps $(x,n)$ to $f_n(x)$.
If $m=\min Q_x$ exists, then $\undef\diff_{n<\om}[f_n/A_n](x)=f_m(x)=k(x,\min Q_x)$.
If no such $m$ exists, $\undef\diff_{n<\om}[f_n/A_n](x)$ is undefined.
This verifies the assertion.
\end{proof}

\begin{prop}
$\Sigma^1_1\lnp$ is $\undef D^\ast_\om(\Sigma^0_1/\Pi^1_1)$-complete.
\end{prop}

\begin{proof}
To see that $\Sigma^1_1\lnp$ is in $\undef D^\ast_\om(\Sigma^0_1/\Pi^1_1)$, define $B_n=\{x:(\forall k<n)\;k\in P_x\}$, and consider $c_n\colon x\mapsto n$.
Then, $(B_n)_{n<\om}$ is an decreasing sequence of $\tpbf{\Pi}^1_1$ sets.
Put $S_x=\om\setminus P_x$, and then one can see that $\min S_x=\max\{n<\om:x\in B_n\}$ whenever $S_x\not=\emptyset$.
Therefore, $\undef\diff^\ast_{n<\om}[c_n/B_n]$ is a realizer for $\Sigma^1_1\lnp$.
To show the completeness, assume that a sequence $(B_n,f_n)$ of pairs of $\tpbf{\Pi}^1_1$ sets and continuous functions is given.
Let us consider the inner reduction $h$ which maps $x$ to a $\delta_{\Sigma^1_1}$-name of $U_x=\{n\in\om:x\not \in B_n\}$, and the outer reduction $k$ which maps $(x,n+1)$ to $f_n(x)$, where $k(x,0)\uparrow$.
Note that $\min U_x=m+1$ if and only if $\{n<\om:x\in B_n\}=m$ as $(B_n)_{n<\om}$ is decreasing.
If $\min U_x>0$, say $\min U_x=m+1$, then $\undef\diff_{n<\om}^\ast[f_n/B_n](x)=f_m(x)=k(x,\min U_x)$.
If no such $m$ exists, $\undef\diff_{n<\om}^\ast[f_n/B_n](x)$ is undefined.
This verifies the assertion.
\end{proof}

\begin{prop}
$\Pi^1_1\lnp\equiv_{\sf W}\#\Sigma^1_1$ and $\Sigma^1_1\lnp\equiv_{\sf W}\#\Pi^1_1$.
\end{prop}

\begin{proof}
Given $A\subseteq\om$, define $A^\ast=\{n\in\om:(\forall m<n)\;m\not\in A\}$.
Clearly, $\min A=\#A^\ast$.
If $A$ is $\tpbf{\Pi}^1_1$ then $A^\ast$ is $\tpbf{\Sigma}^1_1$, and moreover $A\mapsto A^\ast\colon\tpbf{\Pi}^1_1(\om)\to\tpbf{\Sigma}^1_1(\om)$ is computable, that is, given a $\tpbf{\Pi}^1_1$-code of $A$, one can effectively find a $\tpbf{\Sigma}^1_1$-code of $A^\ast$.
Similarly, if $A$ is $\Sigma^1_1$ then $A^\ast$ is $\Pi^1_1$, and moreover $A\mapsto A^\ast\colon\tpbf{\Sigma}^1_1(\om)\to\tpbf{\Pi}^1_1(\om)$ is computable.
Thus, the inner reduction $A\mapsto A^\ast$ witnesses that $\Pi^1_1\lnp\leq_{\sf W}\#\Sigma^1_1$ and $\Sigma^1_1\lnp\leq_{\sf W}\#\Pi^1_1$.

For the converse direction, assume that a ${\Sigma}^1_1$ set $A\subseteq\om$ is given.
If $A$ is finite, then this fact is witnessed at some stage $<\ck$ since $(\exists n)(\forall m>n)\;m\not\in A[s]$ is a $\Delta^1_1$ property, where $A[s]$ is the stage $s$ hyp-approximation of $A$.
Here, recall that $\Sigma^1_1$ is a higher analogue of ``co-c.e.,'' so $(A[s])_{s\in\om}$ is a co-enumeration of $A$, that is, $s<t$ implies $A[s]\supseteq A[t]$.
At each stage $s$, check if $A[s]$ is finite.
If so, enumerate $\#A[s]$ into $B$.
Then, one can easily see $\min B=\#A$.
Moreover, given a $\Sigma^1_1$-code of $A$, one can easily find a $\Pi^1_1$-code of $B$.
This argument can be uniformly relativizable.
Thus, the inner reduction $A\mapsto B$ witnesses that $\#\Sigma^1_1\leq_{\sf W}\Pi^1_1\lnp$.

Assume that a $\Pi^1_1$ set $B\subseteq\om$ is given.
If we see that the $n$th element is enumerated into $B$, i.e., $\#B[s]\geq n$, then co-enumerate $[0,n)$ from $A$.
Then, $\min A=\#B$.
Given a $\Pi^1_1$-code of $B$, one can easily find a $\Sigma^1_1$ code of $A$.
This argument can be uniformly relativizable.
The inner reduction $B\mapsto A$ witnesses that $\#\Pi^1_1\leq_{\sf W}\Sigma^1_1\lnp$.
\end{proof}

One can also consider the least number principle on a well-ordered set.
For a countable ordinal $\alpha$, let $\leq_\alpha$ be a well-order on $\N$ whose order type is $\alpha$.
Then, we use $\Gamma$-${\sf LNP}_\alpha$ to denote the least number principle with respect to $\leq_\alpha$; that is, $\Gamma$-${\sf LNP}_\alpha(A)$ is defined as the $\leq_\alpha$-smallest element of $A$ if it exists.
As in the above argument, one can observe that $\Pi^1_1$-${\sf LNP}_{\alpha}$ and $\Sigma^1_1$-${\sf LNP}_{\alpha}$ correspond to  $\undef D_\alpha(\tpbf{\Sigma}^0_1/\tpbf{\Pi}^1_1)$ and $\undef D^\ast_\alpha(\tpbf{\Sigma}^0_1/\tpbf{\Pi}^1_1)$, respectively.
This idea leads to our solution to Question \ref{four-main-question}.

\subsection{Fournier's problem}\label{sec:Fournier-problem}

The increasing difference hierarchy can be defined by the combination of the parity function and the least number principle on countable well-orders.
Recall that the parity function ${\sf par}\colon{\sf Ord}\to 2$ returns $1$ if a given input is odd; otherwise, returns $0$.
For a countable ordinal $\eta$, let $(A_\xi)_{\xi<\eta}$ be an increasing sequence of subsets of $\om^\om$, and put $A_\eta=\om^\om$.
Then, it is not hard to check the following:
\[\diff_{\xi<\eta}A_\xi=\Big\{x\in\om^\om:{\sf par}\big(\min\{\alpha\leq\eta:x\in A_\alpha\}\big)\not={\sf par}(\eta)\Big\}.\]

Similarly, if $(B_\xi)_{\xi<\eta}$ is an decreasing sequence of subsets of $\om^\om$, then
\[
\left(\diffd_{\xi<\eta}B_\xi\right)(x)=
\begin{cases}
1&\mbox{ if }{\sf par}(\max\{\xi<\eta:x\in B_\xi\})=0,\\
0&\mbox{ if $\max\{\xi<\eta:x\in B_\xi\}$ does not exist}.
\end{cases}
\]

The {\em $\Pi^1_1$-least number principle} on a well-ordered set $(\om,\preceq)$ states that any nonempty $\Pi^1_1$ set $P\subseteq\om$ has the $\preceq$-smallest element.
We represent the $\Pi^1_1$-least number principle as a function as in Section \ref{sec:Weihrauch}.
Here, recall that we have a total representation $\delta_{\Pi^1_1}$ of $\tpbf{\Pi}^1_1(\N)$.
A $\delta_{\Pi^1_1}$-name is often called a {\em $\bfP^1_1$-code}.
Let us use $P_x$ to denote the subset of $\om$ whose $\bfP^1_1$-code is $x$, i.e., $P_x=\delta_{\Pi^1_1}(x)$.
For $y\in{\sf WO}$ and $A\subseteq\N$, we define $\min_yA$ as the $\leq_y$-least element of $A$, i.e., $a=\min_yA$ if and only if $a\in A$ and $b\not\in A$ for any $b<_ya$.

To be more precise, we define $\tpbf{\Pi}^1_1\lnpwo$ as the partial function which, given a $\tpbf{\Pi}^1_1$-code $x$ of $P\subseteq \N$ and a well-order $y=(\N,\leq_y)$, returns the $\leq_y$-smallest element of $P$ whenever $P$ is nonempty, that is,
\[{\rm dom}(\tpbf{\Pi}^1_1\lnpwo)=\{(x,y):P_x\not=\emptyset\mbox{ and }y\in{\sf WO}\},\]
\[\tpbf{\Pi}^1_1\lnpwo(x,y)={\rm min}_yP_x.\]

We consider totalizations of $\tpbf{\Pi}^1_1\lnpwo$.
For each $c\in\om$, define $c\ast\tpbf{\Pi}^1_1\lnpwo$ as follows:
\[
(c\ast\tpbf{\Pi}^1_1\lnpwo)(x,y)=
\begin{cases}
{\rm min}_{y} P_x &\mbox{ if }P_x\not=\emptyset\mbox{ and }y\in{\sf WO},\\
c&\mbox{ otherwise.}
\end{cases}
\]

Note that, contrary to Section \ref{sec:Weihrauch}, we deal with a realizer (i.e., a function on codes) rather than a function between represented spaces.
This ensures that $c\ast\tpbf{\Pi}^1_1\lnpwo$ is a total $\N$-valued function on $\om^\om$.
However, to discuss the Wadge degree, it must be restricted to a two-valued function.
To simplify our argument, we assume that $c=0$.
Then define $\tpbf{\Pi}^1_1\lnpwo^{\upto 2}$ as follows:
\[
(\tpbf{\Pi}^1_1\lnpwo^{\upto 2})(x,y)=
\begin{cases}
{\sf par}_y({\rm min}_{y} P_x) &\mbox{ if }P_x\not=\emptyset\mbox{ and }y\in{\sf WO},\\
0&\mbox{ otherwise.}
\end{cases}
\]

Here, ${\sf par}_y(n)$ is the parity of the $\leq_y$-rank of $n$.
Then, $\tpbf{\Pi}^1_1\lnpwo^{\upto 2}$ is a two-valued function on $\om^\om$.



\begin{prop}\label{prop:lnpwo-non-approximable}
For any countable ordinal $\eta$, every $D_\eta(\tpbf{\Pi}^1_1)$ set is Wadge reducible to $\tpbf{\Pi}^1_1\lnpwo^{\upto 2}$.
\end{prop}

\begin{proof}
Assume that $\eta$ is even.
Fix a well-order $\leq_\eta$ on $\om$ whose order type is $\eta$, and put $\bar{\eta}=(\om,\leq_\eta)$.
Let $A=\diff_{\xi<\eta}A_\xi$ be a $D_\eta(\tpbf{\Pi}^1_1)$ set.
Then, the set $Q=\{(x,n):x\in A_{|n|_\eta}\}$ is $\tpbf{\Pi}^1_1$, where recall that $|n|_\eta$ is the $\leq_\eta$-rank of $n\in\N$ (see Section \ref{preliminaries}).
Thus, one can find a continuous function $\theta$ which, given $x$, returns a $\tpbf{\Pi}^1_1$-code of $Q_x=\{n\in\N:x\in A_{|n|_\eta}\}$.
We claim that $x\mapsto(\theta(x),\bar{\eta})$ is a Wadge reduction witnessing $A\leq_{\sf W}\tpbf{\Pi}^1_1\lnpwo^{\upto 2}$.
Since $\eta$ is even, $x\in A$ if and only if $\min\{\xi\leq\eta:x\in A_\xi\}$ is odd if and only if $\min_\eta(\{n\in\N:x\in A_{|n|_\eta}\})$ is odd if and only if $\tpbf{\Pi}^1_1\lnpwo^{\upto 2}(\theta(x),\bar{\eta})=1$.
This verify the claim.
The case where $\eta$ is odd can be proved in almost the same way.
\end{proof}

As a consequence, $c\ast\tpbf{\Pi}^1_1\lnpwo$ is not hyp-computable with finite mind-changes along any countable ordinal (since the hierarchy $(D_\eta(\tpbf{\Pi}^1_1))_{\eta<\om_1}$ does not collapse).
On the other hand, it is intuitively clear that $c\ast\tpbf{\Pi}^1_1\lnpwo$ is hyp-computable with finite mind-changes.
Indeed, it is hyp-computable with finite mind-changes along the uncountable ordinal $\om_1+1$.
To see this, let $(x,y)$ be an input.
Begin with the guess $c$ and ordinal counter $\om_1<\om_1+1$.
If $y$ is found to be ${\sf WO}$, then change the ordinal counter to the order type $|y|$ of $y$, which is smaller than $\om_1$.
When something is first enumerated into $P_x$, we guess the current $\leq_y$-least element $n\in P_x$ as a correct answer, and change the ordinal counter to $|n|_y<|y|$.
If some number which is $\leq_y$-smaller than the previous guess is enumerated into $P_x$, then change the guess as above.
Continue this procedure.
This algorithm eventually guesses the correct output of $c\ast\tpbf{\Pi}^1_1\lnpwo(x,y)$.
Clearly, this procedure is hyp-computable with finite mind-changes along $\om_1+1$.
Thus, we only need to formalize this argument as a $D_\om^\ast(\tpbf{\Pi}^1_1)$ set.

\begin{prop}\label{prop:lnpwo-approximable}
$\tpbf{\Pi}^1_1\lnpwo^{\upto 2}\in D_\om^\ast(\tpbf{\Pi}^1_1)$.
\end{prop}

\begin{proof}
Define $B_n$ as the set of all $(x,y)$ such that the parity (w.r.t.~the $\leq_y$-rank) of the $\leq_y$-least element of $P_x$ changes at least $n$ times under the cannonical hyp-computable guessing process.
In other words, $B_n$ is the set of all $(x,y)$ satisfying the following conditions:
\begin{align*}
y\in{\sf WO}\;\land\;(\exists s_1<\dots<s_n)\;&\big[{\sf par}_y({\rm min}_y P_x[s_1])=1\\
&\land\;(\forall i<n)\;{\sf par}_y({\rm min}_y P_x[s_i])\not={\sf par}_y({\rm min}_y P_x[s_{i+1}])\big].
\end{align*}

The standard hyperarithmetical quantification argument shows that $B_n$ is $\Pi^1_1$ since we only need to search for $x$-computable ordinals $s_i$.
To be more precise, first recall that the condition $n\in P_x$ is equivalent to ${\bf o}_P(n,x)\in{\sf WO}$.
In this case, ${\bf o}_P(n,x)$ is an $x$-computable well-order since ${\bf o}_P$ is computable.
Similarly, the condition $a=\min_y P_x[s]$ is equivalent to that $|{\bf o}_P(a,x)|<s$ and $|{\bf o}_P(b,x)|\geq s$ for any $b<_ya$.
This is a $\Delta^1_1$ condition on the $\Pi^1_1$ assumption that $s$ is an ordinal.
Putting it all together, the condition $(x,y)\in B_n$ can be written as follows:
\begin{align*}
y\in{\sf WO}\;\land\;(&\exists e_1,\dots,e_n)\;\big[(\forall i\leq n)\;\varphi_{e_i}^x\in{\sf WO}\;\land\;|\varphi_{e_1}^x|<\dots<|\varphi_{e_n}^x|\\
&\;\land\;(\exists a_1,\dots,a_n)[{\sf par}_y(a_1)=1\;\land\;(\forall i<n)\;{\sf par}_y(a_i)\not={\sf par}_y(a_{i+1})]\\
&\;\land \;(\forall i\leq n)\;[|{\bf o}_P(a_i,x)|<|\varphi_{e_i}^x|
\land\;(\forall b)\;(b<_ya_i\;\to\;|{\bf o}_P(b,x)|\geq|\varphi_{e_i}^x|)
]
\big]
\end{align*}

This only involves number quantification (with some $\Pi^1_1$ sets), so this property is $\Pi^1_1$.
It is clear that $(B_n)_{n\in\N}$ is decreasing.
Given $(x,y)$, let $n$ be the largest number such that $(x,y)\in B_n$ with witnesses $s_1<\dots<s_n$.
Then, we have ${\sf par}_y(\min_y P_x)={\sf par}_y(\min_y P_x[s_n])$; otherwise, we must find $s_{n+1}>s_n$ such that ${\sf par}_y(\min_y P_x[s_{n+1}])\not={\sf par}_y(\min_y P_x[s_n])$, which is impossible by the maximality of $n$.
Put $p_i={\sf par}_y(\min_y P_x[s_i])$.
Then, since $p_1=1$ and $p_i\not=p_{i+1}$, we have $p_i={\sf par}(i)$, and therefore, ${\sf par}_y(\min_y P_x)=p_n={\sf par}(n)$.
Consequently, if $n$ is the largest number such that $(x,y)\in B_n$ then ${\sf par}_y(\min_y P_x)={\sf par}(n)$.
Moreover, if there is no such an $n$ then $y\not\in{\sf WO}$ or $P_x$ is empty.
This shows that
\[
\tpbf{\Pi}^1_1\lnpwo^{\upto 2}(x,y)=\left(\diffd_{n<\om}B_n\right)(x,y)=
\begin{cases}
{\sf par}(n)&\mbox{ if }n=\max\{n:(x,y)\in B_n\},\\
0&\mbox{ if there is no such an $n$}.
\end{cases}
\]

Hence, $\tpbf{\Pi}^1_1\lnpwo^{\upto 2}\in D_\om^\ast(\tpbf{\Pi}^1_1)$.
\end{proof}

Consequently, $\tpbf{\Pi}^1_1\lnpwo^{\upto 2}$ is contained in the $\om$-th level of the decreasing difference hierarchy, but not in the increasing difference hierarchy.
This solves Fournier's question:

\begin{proof}[Proof of Theorem \ref{thm:solution-to-Fournier}]
By Proposition \ref{prop:lnpwo-approximable}, $\tpbf{\Pi}^1_1\lnpwo^{\upto 2}$ belongs to the $\om$-th level of the decreasing difference hierarchy.
If $\tpbf{\Pi}^1_1\lnpwo^{\upto 2}\in {\sf Diff}(\tpbf{\Pi}^1_1)$ would hold, then $\tpbf{\Pi}^1_1\lnpwo^{\upto 2}\in D_\eta(\tpbf{\Pi}^1_1)$ for some $\eta<\om_1$.
However, by Proposition \ref{prop:lnpwo-non-approximable}, every $D_{\eta+1}(\tpbf{\Pi}^1_1)$ set is Wadge reducible to $\tpbf{\Pi}^1_1\lnpwo^{\upto 2}$.
A simple diagonalization argument shows the existence of a $D_{\eta+1}(\tpbf{\Pi}^1_1)$ set which is not Wadge reducible to a $D_\eta(\tpbf{\Pi}^1_1)$ set.
This implies a contradiction; hence, $\tpbf{\Pi}^1_1\lnpwo^{\upto 2}\not\in {\sf Diff}(\tpbf{\Pi}^1_1)$.
\end{proof}

\subsection{Beyond the decreasing difference hierarchy}\label{sec:beyond-decreasing}

The decreasing difference hierarchy over $\tpbf{\Pi}^1_1$ sets occupies a very small part of the smallest $\sigma$-algebra including all $\tpbf{\Pi}^1_1$ sets.
Let us turn our attention to the first level of the $\sigma$-algebra.

\begin{definition}[see e.g.~Becker {\cite[Page 719]{Becker88}}]
For a pointclass $\Gamma$, let $\tpbf{\Sigma}^0_1(\Gamma)$ be the smallest family including all $\Gamma$ sets and closed under countable union, finite intersection, and continuous preimage.
A set $A$ is in $\tpbf{\Delta}^0_1(\Gamma)$ if both $A$ and its complement is contained in $\tpbf{\Sigma}^0_1(\Gamma)$.
\end{definition}

Higher limit lemma \cite[Proposition 6.1]{BGM17} states that ${\Delta}^0_1(\Sigma^1_1\cup\Pi^1_1)$ is equivalent to hyp-computability with ordinal mind-changes.
Note that this result does not imply that ${\Delta}^0_1(\Sigma^1_1\cup\Pi^1_1)$ is equivalent to ${\rm Diff}^\ast(\bfP^1_1)$.
This is because ${\rm Diff}^\ast({\Pi}^1_1)$ corresponds to hyp-computability with ordinal mind-changes involving some countable ordinal which bounds the number of mind-changes for all inputs, while in the case of ${\Delta}^0_1(\Sigma^1_1\cup\Pi^1_1)$, the number of mind-changes can be different for each input, and it is not always possible to give their upper bound by a single countable ordinal.
Indeed, using a similar argument as above, we show that ${\rm Diff}^\ast(\tpbf{\Pi}^1_1)$ is a proper subclass of $\tpbf{\Delta}^0_1(\tpbf{\Sigma}^1_1\cup\tpbf{\Pi}^1_1)$.

The {\em $\Sigma^1_1$-least number principle} on a well-ordered set $(\om,\preceq)$ states that any nonempty $\Sigma^1_1$ set $S\subseteq\om$ has the $\preceq$-smallest element.
Let us use $S_x$ to denote the subset of $\om$ whose $\tpbf{\Sigma}^1_1$-code is $x$, i.e., $S_x=\om\setminus P_x$.
One can define the totalization $c\ast\bfS^1_1\lnpwo$ of the partial $\om$-valued function $\bfS^1_1\lnpwo$ as above.
Then define its two-valued restriction $\tpbf{\Sigma}^1_1\lnpwo^{\upto 2}$ as follows:
\[
(\tpbf{\Sigma}^1_1\lnpwo^{\upto 2})(x,y)=
\begin{cases}
{\sf par}_y({\rm min}_{y} S_x) &\mbox{ if }S_x\not=\emptyset\mbox{ and }y\in{\sf WO},\\
0&\mbox{ otherwise.}
\end{cases}
\]

Then, $\tpbf{\Sigma}^1_1\lnpwo^{\upto 2}$ is a two-valued total function on $\om^\om$.

\begin{prop}\label{prop:lnpwo-non-approximable2}
For any countable ordinal $\eta$, every $D_\eta^\ast(\tpbf{\Pi}^1_1)$ set is Wadge reducible to $\tpbf{\Sigma}^1_1\lnpwo^{\upto 2}$.
\end{prop}

\begin{proof}
Fix a well-order $\leq_\eta$ on $\om$ whose order type is $\eta$, and put $\bar{\eta}=(\om,\leq_\eta)$.
Let $B=\diff_{\xi<\eta}B_\xi$ be a $D^\ast_\eta(\tpbf{\Pi}^1_1)$ set.
Then, the set $U=\{(x,n):x\not\in B_{|n|_\eta}\}$ is $\tpbf{\Sigma}^1_1$.
Thus, one can find a continuous function $\theta$ which, given $x$, returns a $\tpbf{\Sigma}^1_1$-code of $U_x=\{n\in\N:x\not\in B_{|n|_\eta}\}$.
We claim that $x\mapsto(\theta(x),\bar{\eta})$ is a Wadge reduction witnessing $B\leq_{\sf W}\tpbf{\Sigma}^1_1\lnpwo^{\upto 2}$.
If $\gamma=\max\{\xi\leq\eta:x\in B_\xi\}$ exists, $x\in B$ if and only if $\gamma$ is even.
In this case, the $\leq_\eta$-rank of $\min_\eta U_x$ is $\gamma+1$, which is odd, and therefore, $\tpbf{\Sigma}^1_1\lnpwo^{\upto 2}(\theta(x),\bar{\eta})=1$.
If no such a $\gamma$ exists, then $x\not\in B$, and the $\leq_\eta$-rank of $\min_\eta U_x$ is a limit ordinal, which is even, and therefore, $\tpbf{\Sigma}^1_1\lnpwo^{\upto 2}(\theta(x),\bar{\eta})=0$.
In either case, we have $B(x)=\tpbf{\Sigma}^1_1\lnpwo^{\upto 2}(\theta(x),\bar{\eta})$.
\end{proof}

As a consequence, $c\ast\tpbf{\Sigma}^1_1\lnpwo$ is not hyp-computable with fixed countable ordinal mind-changes (since the hierarchy $(D^\ast_\eta(\tpbf{\Pi}^1_1))_{\eta<\om_1}$ does not collapse).
On the other hand, it is intuitively clear that $c\ast\tpbf{\Sigma}^1_1\lnpwo$ is hyp-computable with ordinal mind-changes.
To see this, let $(x,y)$ be an input, and begin with the guess $c$.
If $y$ is found to be ${\sf WO}$, we guess the current $\leq_y$-least element $n\in S_x$ as a correct answer.
If all numbers which are $\leq_y$-smaller than or equal to the previous guess is removed from $S_x$, then change the guess as above.
Continue this procedure.
This algorithm eventually guesses the correct output of $c\ast\tpbf{\Sigma}^1_1\lnpwo(x,y)$.
Clearly, this procedure is hyp-computable with ordinal mind-changes.
Thus, we only need to formalize this argument as a $\tpbf{\Delta}^0_1(\bfS^1_1\cup\bfP^1_1)$ set.

\begin{prop}\label{prop:lnpwo-approximable2}
$\tpbf{\Sigma}^1_1\lnpwo^{\upto 2}\in\tpbf{\Delta}^0_1(\bfS^1_1\cup\bfP^1_1)$.
\end{prop}

\begin{proof}
We first consider the set of stages at which the least value of $S_x$ changes later.
In other words, define $W$ as follows:
\begin{align*}
(x,y,s)\in W\iff &y,s\in{\sf WO}\;\land\exists t\leq_Tx\oplus y\;(t\in{\sf WO}\\
&\land\;|t|>|s|\;\land\;({\rm min}_yS_x[s]<{\rm min}_yS_x[t]\;\lor\;S_x[t]=\emptyset)).
\end{align*}

It is easy to see that $W$ is $\Pi^1_1$.
We claim that if $S_x\not=\emptyset$ and $y\in{\sf WO}$ then there exists an ordinal $s\leq_Tx\oplus y$ such that ${\rm min}_yS_x[s]={\rm min}_yS_x$.
To see this, assume that $a={\rm min}_yS_x$.
Then, for any $b<_ya$ there exists an $x$-computable ordinal $s_b\in{\sf WO}$ such that $b\in P_x[s_b]$.
Note that $A=\{b\in\om:b<_ya\}$ is a $\Delta^1_1(y)$ set.
Then consider the map $b\mapsto e_b$, where $e_b$ is an $x$-computable index of such $s_b$, which is a $\Delta^1_1(x\oplus y)$ function.
The usual $\Sigma^1_1$-bounding argument (i.e., the relativized Spector boundedness theorem) ensures that $s=\sup\{s_b:b<_ya\}$ is an $(x\oplus y)$-computable ordinal.
This verifies the claim.

This claim shows that, for any $y\in{\sf WO}$ and $i<2$, the statement $S_x\not=\emptyset$ and ${\sf par}_y(\min_yS_x)=i$ holds if and only if there exists an ordinal $s\leq_Tx\oplus y$ such that $(x,y,s)\not\in W$ and ${\sf par}_y(\min_yS_x[s])=i$.
Therefore,
\begin{align*}
\tpbf{\Sigma}^1_1\lnpwo^{\upto 2}(x,y)=1
\iff
&y\in{\sf WO}\;\land\;\exists s\leq_Tx\oplus y\\
&(s\in{\sf WO}\;\land\;(x,y,s)\not\in W\;\land\;{\sf par}_y({\rm min}_yS_x[s])=1),
\end{align*}
and similarly
\begin{align*}
\tpbf{\Sigma}^1_1\lnpwo^{\upto 2}(x,y)=0
\iff
y&\not\in{\sf WO}\ \lor\ \forall s\leq_Tx\oplus y\;(x,y,s)\in W\\
\lor\ &[y\in{\sf WO}\;\land\;\exists s\leq_Tx\oplus y\\
&(s\in{\sf WO}\;\land\;(x,y,s)\not\in W\;\land\;{\sf par}_y({\rm min}_yS_x[s])=0)].
\end{align*}

The former formula is clearly $\tpbf{\Sigma}^0_1(\bfS^1_1\cup\bfP^1_1)$.
The latter formula contains a universal quantification, but the first line and the second line are separated, and the subformula ``$\forall s\leq_Tx\oplus y\;(x,y,s)\in W$'' is $\Pi^1_1$.
Hence, the latter formula is also $\tpbf{\Sigma}^0_1(\bfS^1_1\cup\bfP^1_1)$.
Note also that the first line in the latter formula is equivalent to the statement that either $y\not\in{\sf WO}$ or $S_x=\emptyset$ holds.
Consequently, we get $\tpbf{\Sigma}^1_1\lnpwo^{\upto 2}\in\tpbf{\Delta}^0_1(\bfS^1_1\cup\bfP^1_1)$.
\end{proof}

Consequently, $\tpbf{\Sigma}^1_1\lnpwo^{\upto 2}$ is contained in the first $\tpbf{\Delta}$-level of the $\sigma$-algebra containing $\bfP^1_1$ sets, but not in the decreasing difference hierarchy.
That is, $\tpbf{\Sigma}^1_1\lnpwo^{\upto 2}$ witnesses the properness of the inclusion ${\sf Diff}^\ast(\tpbf{\Pi}^1_1)\subsetneq\tpbf{\Delta}^0_1(\tpbf{\Pi}^1_1\cup\tpbf{\Sigma}^1_1)$.
This solves Question \ref{question:beyond-difference}:

\begin{proof}[Proof of Theorem \ref{thm:beyond-difference}]
By Proposition \ref{prop:lnpwo-approximable2}, $\tpbf{\Sigma}^1_1\lnpwo^{\upto 2}$ belongs to $\tpbf{\Delta}^0_1(\bfS^1_1\cup\bfP^1_1)$.
If $\tpbf{\Sigma}^1_1\lnpwo^{\upto 2}\in {\sf Diff}^\ast(\tpbf{\Pi}^1_1)$ would hold, then $\tpbf{\Sigma}^1_1\lnpwo^{\upto 2}\in D_\eta^\ast(\tpbf{\Pi}^1_1)$ for some $\eta<\om_1$.
However, by Proposition \ref{prop:lnpwo-non-approximable2}, every $D_{\eta+1}^\ast(\tpbf{\Pi}^1_1)$ set is Wadge reducible to $\tpbf{\Sigma}^1_1\lnpwo^{\upto 2}$.
A simple diagonalization argument shows the existence of a $D_{\eta+1}^\ast(\tpbf{\Pi}^1_1)$ set which is not Wadge reducible to a $D_\eta^\ast(\tpbf{\Pi}^1_1)$ set.
This implies a contradiction; hence, $\tpbf{\Sigma}^1_1\lnpwo^{\upto 2}\not\in {\sf Diff}^\ast(\tpbf{\Pi}^1_1)$.
\end{proof}

\section{The $\om$-th level of the decreasing difference hierarchy}\label{sec:4}

\subsection{$\om_1$-prewellordered coproduct}

Next, we analyze the structure of $\Delta(D^\ast_\om(\tpbf{\Pi}^1_1))$ sets.
We first show the following useful characterization of $\Delta(D^\ast_\om(\tpbf{\Pi}^1_1))$ sets.

\begin{prop}
A set $P\subseteq\om^\om$ belongs to $\Delta(D^\ast_\om(\tpbf{\Pi}^1_1))$ if and only if there exists an infinite decreasing sequence $(P_n)_{n\in\om}$ of $\tpbf{\Pi}^1_1$ sets such that $\bigcap_{n<\om}P_n=\emptyset$ and $P=\diff^\ast_{n<\om}P_n$.
\end{prop}

\begin{proof}
If $P\in\Delta(D^\ast_\om(\tpbf{\Pi}^1_1))$, then there exist infinite decreasing sequences $(A_n)_{n\in\om}$ and $(B_n)_{n\in\om}$ of $\tpbf{\Pi}^1_1$ sets such that $P=\diff^\ast_nA_n$ and $\neg P=\diff^\ast_nB_n$.
Note that $x\in\bigcap_n A_n$ implies $x\not\in\diff^\ast_nA_n$, so $x\in P$, and similarly, $x\in\bigcap_nB_n$ implies $x\not\in P$.
Hence, $\bigcap_nA_n\cap\bigcap_nB_n=\emptyset$.
Then, define $P_n=A_n\cap B_{n+1}$.
Then, $\bigcap_nP_n\subseteq\bigcap_n A_n\cap\bigcap_nB_n=\emptyset$.
Moreover, it is not hard to check that $P=\diff^\ast_nP_n$.

For the converse direction, let $(P_n)$ be such that $\bigcap_nP_n=\emptyset$ and $P=\diff^\ast_nP_n$.
Then, define $A_n=P_n$, $B_0=\om^\om$, and $B_n=P_{n+1}$.
It is easy to check that $P=\diff^\ast_nA_n$ and $\neg P=\diff^\ast_nB_n$.
\end{proof}

As we have already mentioned, the class $\Delta(D^\ast_\om(\tpbf{\Pi}^1_1))$ corresponds to hyp-computability with finite mind-changes.
As usual, the process of mind-changes can be represented by a well-founded tree.
We describe the details below.

Under {\sf AD}, recall that every nonselfdual subset of $\om^\om$ is Wadge equivalent to a subset of $2^\om$; see e.g.~\cite[Lemma 1.5]{Kih19}, and any selfdual set is Wadge equivalent to the join of countably many nonselfdual set; see e.g.~\cite{AnLo12}.
Therefore, one may assume that everything is a subset of the $\sigma$-compact space $\C=\om\times 2^\om$.
Recall that ${\sf WO}\subseteq 2^{\om\times\om}\simeq 2^\om$ is the set of all well-orders on $\om$.

Let $(P_n)_{n\in\om}$ be an infinite decreasing sequence of $\tpbf{\Pi}^1_1$ sets in $\C$ such that $\bigcap_{n\in\om}P_n=\emptyset$.
Since ${\sf WO}$ is $\tpbf{\Pi}^1_1$-complete, there exists a continuous function $\theta_n$ witnessing $P_n\leq_{\sf W}{\sf WO}$.
Then, define $P_n[c,\alpha]=\theta_n^{-1}\{\alpha\}\cap(\{c\}\times 2^\om)$ for any $c\in\om$ and $\alpha\in{\sf WO}$.
Clearly, $P_n[c,\alpha]$ is compact, and we have $P_n=\bigcup_{c,\alpha}P_n[c,\alpha]$.
Hereafter we omit $c$ to simplify the notation.

\begin{definition}\label{def:tree-system}
Given such a sequence $P=(P_n)_{n\in\om}$, one can define a system on a labeled ${\sf WO}$-branching well-founded tree $T_P\subseteq{\sf WO}^{<\om}$ as follows:

To each node $\sigma$ of $T_P$ of length $n$, assign the sequence $(P_n[\alpha])_{\alpha\in{\sf WO}}$.
If the length $n$ is even, then the node is labeled by $0$; otherwise, it is labeled by $1$.
The domain on $\sigma$ is defined as $Q_\sigma:=\bigcap_{m<n}P_m[\sigma(m)]$.
We add the $\alpha$-th immediate successor of $\sigma$ whenever $Q_\sigma\cap P_n[\alpha]$ is nonempty.
In other words, define $T_P=\{\sigma\in{\sf WO}^{<\om}:Q_\sigma\not=\emptyset\}$.
\end{definition}

Note that if $x\in Q_\sigma$ then $\sigma=\langle\theta_0(x),\theta_1(x),\dots,\theta_{|\sigma|-1}(x)\rangle$ since $x\in P_m[\sigma(m)]$ if and only if $\theta_m(x)=\sigma(m)$.

\begin{obs}
For $P=(P_n)_{n<\om}$, if $\bigcap_{n\in\om}P_n=\emptyset$
then $T_P$ is well-founded.
\end{obs}

\begin{proof}
If $T_P$ has an infinite path $p\in{\sf WO}^\om$, then for any $n$, the compact set $\bigcap_{m\leq n}P_m[p(m)]$ is nonempty.
Therefore, by compactness, the whole intersection $\bigcap_{n<\om}P_n[p(n)]\subseteq\bigcap_nP_n$ is also nonempty, which contradicts our assumption on $(P_n)_{n\in\om}$.
\end{proof}

Note also that $T_P$ is Borel on ${\sf WO}^{<\om}$.
One can recover the information on $\diff^\ast_nP_n$ in the following manner.

\begin{obs}\label{obs:tree-mind-change-x}
Let $P$ and $(\theta_n)$ be as above.
For $x\in\om^\om$, define $\sigma_x$ as the maximal initial segment of $(\theta_\ell(x))_{\ell<\om}$ which is contained in $T_P$.
Then, $x\in\diff^\ast_{n<\om}P_n$ if and only if $\sigma_x$ is labeled by $1$.
\end{obs}

\begin{proof}
Assume that $\sigma_x=(\theta_{\ell}(x))_{\ell<k}$.
Then, $x\in Q_{\sigma_x}\subseteq\bigcap_{m<k}P_m$.
Since this is maximal, $\sigma_x':=\sigma_x\fr\theta_k(x)$ is not contained in $T_P$.
If $\theta_k(x)\in{\sf WO}$ then we have $x\in Q_{\sigma_x'}$, so $\sigma_x'$ must be contained in $T_P$.
Hence, $\theta_k(x)\not\in{\sf WO}$, so $x\not\in P_k$.
Therefore, $\max\{\ell:x\in P_\ell\}=k-1$.
Thus, $x\in\diff^\ast_nP_n$ if and only if ${\sf par}(k-1)=0$, so ${\sf par}(k)=1$.
This means that the length of $\sigma_x$ is odd.
In this case, $\sigma_x$ is labeled by $1$.
\end{proof}

In a more inductive manner, one can recover the information of $\diff^\ast_nP_n$.
For $\sigma\in T_P$, inductively define $f_\sigma\colon\om^\om\to\{0,1\}$ as follows:
If a leaf $\rho$ is labeled by $0$, define $f_\rho\colon x\mapsto 0$.
If a leaf $\rho$ is labeled by $1$, define $f_\rho\colon x\mapsto 1$.
If a node $\sigma$ is not a leaf, and is labeled by $i$, define 
\[
f_\sigma(\alpha,x):=i\ast \bigsqcup_{\alpha\in{\sf WO}}f_{\sigma\alpha}(x):=
\begin{cases}
\ f_{\sigma\alpha}(x) &\mbox{ if }\alpha\in{\sf WO},\\
\ i &\mbox{ if }\alpha\not\in{\sf WO}.
\end{cases}
\]

\begin{lemma}\label{lem:from-tree-to-diff}
$x\in\diff^\ast_nP_n\iff f_{\langle\rangle}((\theta_\ell(x))_{\ell<\om})=1$, where $\langle\rangle$ is the empty string.
\end{lemma}

\begin{proof}
Let $\sigma_x=(\theta_m(x))_{m<n}$ is a string as in Observation \ref{obs:tree-mind-change-x}.
Then $n$ be the least number such that $\theta_n(x)\not\in{\sf WO}$.
For $\sigma=\sigma_x$, by the definition of $f_\sigma$, note that
\[f_{\langle\rangle}(\theta_0(x),\theta_1(x),\theta_2(x),\dots)=f_\sigma(\theta_n(x),\theta_{n+1}(x),\dots).\]

If $\sigma$ is labeled by $i$ then $f_\sigma(\theta_n(x),y)=i$ for any $y$ since $\theta_n(x)\not\in{\sf WO}$.
Hence, $\sigma$ is labeled by $i$ if and only if $f_{\langle\rangle}((\theta_\ell(x))_{\ell<\om})=i$.
By Observation \ref{obs:tree-mind-change-x}, $\sigma=\sigma_x$ is labeled by $1$ if and only if $x\in\diff^\ast_nP_n$.
This verifies the claim.
\end{proof}

Thus, $\diff^\ast_nP_n$ is constructed from constant functions and the ${\sf WO}$-indexed coproduct.
To formalize this idea, given a pointclass $\Gamma$, define $\Delta=\Gamma\cap\neg\Gamma$ as usual.

\begin{definition}\label{def:uniform-delta}
We say that $(A_\alpha)_{\alpha\in I}$ is a {\em uniform $\Delta$ collection} if
there are $B,C\in\Gamma$ such that for any $\alpha$ and $z$,
\[\alpha\in I\implies [z\in A_\alpha\iff (\alpha,z)\in B\iff (\alpha,z)\not\in C].\]

We say that a pointclass $\Gamma$ is {\em strictly closed under $\om_1$-prewellordered ($\om_1$-pwo) coproduct} if,
for any uniform $\Delta$ collection $(A_\alpha)_{\alpha\in{\sf WO}}$, we have
\[
\bigsqcup_{\alpha\in{\sf WO}}A_\alpha:=\{(\alpha,x):\alpha\in{\sf WO}\land x\in A_\alpha\}\in\Delta.
\]
\end{definition}

If we identify a set $A\subseteq\om^\om$ with its characteristic function $\chi_A\colon\om^\om\to 2$, then $\bigsqcup_\alpha A_\alpha$ and $0\ast\bigsqcup_\alpha A_\alpha$ are the same.
One can also see that if $\Gamma$ is strictly closed under $\om_1$-pwo coproduct, then we have $1\ast\bigsqcup_{\alpha\in{\sf WO}}A_\alpha\in\Delta$.
To see this, first note that $A\in\Delta$ implies $\neg A\in\Delta$.
Similarly, if $(A_\alpha)_{\alpha\in{\sf WO}}$ is uniformly $\Delta$, so is $(\neg A_\alpha)_{\alpha\in{\sf WO}}$.
Thus,
\[1\ast\bigsqcup_{\alpha\in{\sf WO}}A_\alpha:=\{(\alpha,x):\alpha\not\in{\sf WO}\lor x\in A_\alpha\}=\neg\bigsqcup_{\alpha\in{\sf WO}}(\neg A_\alpha)\in\Delta.\]

\begin{obs}\label{obs:closed-pwo-coprod}
$D_\om^\ast(\tpbf{\Pi}^1_1)$ is strictly closed under $\om_1$-pwo coproduct.
\end{obs}

\begin{proof}
The algorithmic reason for this can be explained as follows:
Given an input $(\alpha,x)$, we have $\alpha\not\in{\sf WO}$ at the first stage, so the learner guesses that $(\alpha,x)\in\bigsqcup_{\alpha\in{\sf WO}}A_\alpha$ is false.
If the learner sees $\alpha\in{\sf WO}$ at some stage, change her mind, and then since $A_\alpha\in\Delta(D_\om^\ast(\tpbf{\Pi}^1_1))$, the learner only needs to simulate a guessing process to answer whether $x\in A_\alpha$ or not with finite mind-changes.

The set-theoretic reason for this is as follows:
Let a pair $(B,C)$ be a $\Delta$-definition of $(A_\alpha)_{\alpha\in{\sf WO}}$ as in Definition \ref{def:uniform-delta}.
It is easy to see that $\bigsqcup_{\alpha\in{\sf WO}}A_\alpha$ and its complement can be written as $\pi_0^{-1}[{\sf WO}]\cap B$ and $\pi_0^{-1}[\neg{\sf WO}]\cup C$, respectively.
Since $D_\om^\ast(\tpbf{\Pi}^1_1)$ is closed under finite union with $\tpbf{\Sigma}^1_1$ sets and finite intersection with $\tpbf{\Pi}^1_1$ sets, both sets belong to $D_\om^\ast(\tpbf{\Pi}^1_1)$.
\end{proof}

A key basic fact on the closure property for $\Delta$ under ${\sf AD}$ is that, if $\Delta$ is closed under something, then it is closed {\em uniformly}, as shown by Becker \cite{Bec84}.
As a special case, we have the following:

\begin{fact}[Becker \cite{Bec84}, {\sf AD}]\label{lem:Becker-uniform}
If $\Gamma$ is strictly closed under $\om_1$-pwo coproduct, then there exists a continuous function which, given a uniform $\Delta$-code of $(A_\alpha)_{\alpha\in{\sf WO}}$, returns a $\Delta$-code of $\bigsqcup_{\alpha\in{\sf WO}}A$.
\end{fact}

The Wedge reducibillity is too fine-grained to handle this level of pointclasses, and for this reason we first deal with a coarser reducibility.
For $A,B\subseteq\om^\om$, we say that $A$ is {\em Borel-Wadge reducible to $B$} ({\em written $A\leq_{\sf BW}B$}) if there exists a Borel function $\theta\colon\om^\om\to\om^\om$ such that, for any $x\in\om^\om$, $x\in A$ if and only if $\theta(x)\in B$.
The Borel-Wadge degrees are semi-well-ordered, and therefore, one can assign a Borel-Wadge rank $|A|_{\sf BW}$ to each set $A\subseteq\om^\om$.
A {\em Borel-Wadge pointclass} is a class of subsets of $\om^\om$ downward closed under Borel-Wadge reducibility, i.e., $A\in\Gamma$ and $B\leq_{\sf BW}A$ implies $B\in\Gamma$.
For basic information on Borel-Wadge reducibility, see Andretta-Martin \cite{AnMa03}.

Now we give a key result connecting the class $D_\om^\ast(\tpbf{\Pi}^1_1)$ and the $\om_1$-pwo coproduct.

\begin{prop}[{\sf AD}]\label{prop:omf-closed-coproduct}
$D_\om^\ast(\tpbf{\Pi}^1_1)$ is a minimal nonselfdual Borel-Wadge pointclass which is strictly closed under $\om_1$-pwo coproduct.
\end{prop}

\begin{proof}
By Observation \ref{obs:closed-pwo-coprod}, $D_\om^\ast(\tpbf{\Pi}^1_1)$ is strictly closed under $\om_1$-pwo coproduct.
Thus, we only need to show the minimality.
Assume that $\Gamma$ is strictly closed under $\om_1$-pwo coproduct.
It suffices to show that $\Delta(D_\om^\ast(\tpbf{\Pi}^1_1))\subseteq\Delta$.
As in Definition \ref{def:tree-system}, any $\diff^\ast_{n<\om}P_n\in D_\om^\ast(\tpbf{\Pi}^1_1)$ can be represented as a system on a labeled ${\sf WO}$-branching tree $T_P$, where $P=(P_n)_{n<\om}$.
Then, assign a function $f_\sigma\colon\om^\om\to 2$ to each node $\sigma\in T_P$ as above, and define $Z_\sigma=f_\sigma^{-1}\{1\}$.
To be precise, if $\rho$ is a leaf then $Z_\rho$ is either $\emptyset$ or $\om^\om$ depending on the label of $\rho$, and if $\sigma\in T_P$ is not a leaf then $Z_\sigma=i\ast\bigsqcup_{\alpha\in{\sf WO}}Z_{\sigma\alpha}$, where $i$ is the label of $\sigma$.

\begin{claim}
$Z_{\sigma}\in\Delta$ for any $\sigma\in T_P$.
\end{claim}

\begin{proof}
By Fact \ref{lem:Becker-uniform}, there exists a continuous function which, given a uniform $\Delta$-code of $(A_\alpha)_{\alpha\in{\sf WO}}$, returns a $\Delta$-code of $i\ast\bigsqcup_{\alpha\in{\sf WO}}A_\alpha$.
We define a partial function $h\colon{\sf WO}^{<\om}\to\om^\om$ such that $h(\sigma)$ is a $\Delta$-code of $Z_\sigma$.
The recursion theorem allows us to use a self-referential definition such as ``let $h(\sigma)$ be a $\Delta$-code of the $\om_1$-pwo coproduct of the $\Delta$-sets $(Z_{\sigma\alpha})_{\alpha\in{\sf WO}}$ coded by $(h(\sigma\alpha))_{\alpha\in{\sf WO}}$.''

To discuss the complexity of $h$, we give the details of the above argument:
Given $\sigma\in{\sf WO}^{<\om}$, first check whether $\sigma$ extends a leaf of $T_P$ or not.
This is a Borel property, so it is doable by a $\tpbf{\Pi}^1_1$-measurable way, and the recursion theorem holds for $\tpbf{\Pi}^1_1$, cf.~Moschovakis \cite[Theorem 7A.2]{mos07}.
If $\sigma$ extends a leaf $\rho$, then $h(\sigma)$ is a $\Delta$-code of either $\emptyset$ or $\om^\om$, depending on the length of the leaf $\rho$.
If $\sigma$ does not extend a leaf, calculate a $\tpbf{\Pi}^1_1$-code of $\alpha\mapsto h(\sigma\alpha)$.
Then, by applying Lemma \ref{lem:Becker-uniform} to this code, we hope to obtain the $\Delta$-code $c$ of of the $\om_1$-pwo coproduct of the $\Delta$ sets coded by $(h(\sigma\alpha))_{\alpha\in{\sf WO}}$, and define $h(\sigma)=c$.
However, the problem is that since $h$ is $\tpbf{\Pi}^1_1$-measurable, it is not immediately guaranteed that $(Z_{\sigma\alpha})_{\alpha\in{\sf WO}}$ is a uniform $\Delta$ collection.

In order to overcome this difficulty, let us notice that $Q_\sigma:=\bigcap_{n<|\sigma|}P_n[\sigma(n)]$ is compact uniformly in $\sigma\in{\sf WO}^{<\om}$ (even in $\sigma\in(\om^\om)^{<\om}$).
In other words, we have a continuous function which, given $\sigma$, returns a $\tpbf{\Pi}^0_1$-code of $Q_\sigma$.
Hence, one can decide whether $\sigma$ extends a leaf by a partial stable Baire-one function $\psi$, where a function $f$ is stable Baire-one if there exists a partial continuous function $\tilde{f}$ such that for any $x\in{\rm dom}(f)$ we have $f(x)=\tilde{f}(n,x)$ for all but finitely many $n$.
In particular, such an $f$ is Baire-one, and therefore, the domain of $f$ can be extended to a Borel set.
The recursion theorem for partial stable Baire-one functions follows from the classical recursion theorem applied to the partial continuous function $\tilde{f}$.

Now, the definition of $h$ is given as follows:
If $\psi(\sigma)=1$ (i.e., $\sigma$ extends a leaf $\rho$), then $h(\sigma)$ is a code of $Z_\sigma$, which is either $\emptyset$ or $\om^\om$, depending on the length of the leaf $\rho$.
Otherwise, if $\sigma\alpha\in{\rm dom}(h)$ and $h(\sigma\alpha)$ is a $\Delta$-code of $Z_{\sigma\alpha}$, for a fixed $\Gamma$-universal set $G$, we have
\[x\in Z_{\sigma\alpha}\iff(\pi_0h(\sigma\alpha),x)\not\in G\iff(\pi_1h(\sigma\alpha),x)\in G.\]

Since $\Gamma$ is a Borel-Wadge pointclass, we have 
\[G^i_\sigma:=\{(\alpha,x):\sigma\alpha\in{\rm dom}(h)\; \&\;(\pi_ih(\sigma\alpha),x)\in G\}\in\Gamma.\]

Moreover, a $\Gamma$-code $c^i_\sigma$ of $G^i_\sigma$ can be uniformly obtained from $\sigma$ and a code of $h$.
This ensures that, whenever $h(\sigma\alpha)$ is defined for all $\alpha\in{\sf WO}$, the collection $(Z_{\sigma\alpha})_{\alpha\in{\sf WO}}$ is uniformly $\Delta$, whose code is given by $(c^0_\sigma,c^1_\sigma)$.
Then, let $\tau_0$ be a partial continuous function obtained by Fact \ref{lem:Becker-uniform} and $\tau_1$ be its dual.
In particular, $\tau_i(c^0_\sigma,c^1_\sigma)$ is a code of $i\ast\bigsqcup_{\alpha\in{\sf WO}}Z_{\sigma\alpha}$.
Then, $h(\sigma)$ is defined as $\tau_i(c^0_\sigma,c^1_\sigma)$, where $i$ is the label of $\sigma$.
If $h(\sigma\alpha)$ is defined as a $\Delta$-code for all $\alpha\in{\sf WO}$, then $h(\sigma)$ is also defined, and gives a code of $Z_\sigma=i\ast\bigsqcup_{\alpha\in{\sf WO}}Z_{\sigma\alpha}$.

The recursion theorem ensures that $h$ is well-defined, and by transfinite recursion, we conclude that $h(\sigma)$ is a $\Delta$-code of $Z_{\sigma}$ for any $\sigma\in T_P$.
\end{proof}

It remains to show that $\diff^\ast_nP_n\in\Delta$.
By Lemma \ref{lem:from-tree-to-diff}, given $x$, we have
\[x\in \diffd_{n<\om}P_n\iff (\theta_0(x),\theta_1(x),\theta_2(x),\dots)\in Z_{\langle\rangle}.\]

Consequently, $\diff^\ast_nP_n\leq_{\sf W}Z_{\langle\rangle}$ via $(\theta_0,\theta_1,\theta_2,\dots)$, and thus $\diff^\ast_nP_n\in\Delta$ by the above claim.
\end{proof}


\subsection{Lower bound}
%
%

A lower bound of the Wadge rank of $\Delta(D^\ast_\om(\tpbf{\Pi}^1_1))$ can be given by an argument explained in Steel \cite[Theorem 1.2]{St81}; see also Fournier \cite[Proposition 5.10]{FoPhD}.

\begin{lemma}[{\sf AD}]\label{lem:steel-fournier-game}
Assume that $\Gamma$ is strictly closed under $\om_1$-pwo coproduct.
Then, the cofinality of the Wadge rank of $\Gamma$ is at least $\om_2$.
\end{lemma}

\begin{proof}
Let $\psi\colon\om_1\to\Delta$ be any function. 
Consider the following Solovay game:
Player I chooses a large countable ordinal $\alpha$ and Player II chooses a $\Delta$ set whose Wadge rank is greater than $\psi(\alpha)$.
More precisely, Player I chooses $\alpha\in\om^\om$ and then Player II chooses $\Gamma$-codes of sets $D$ and $E$.
Player II wins if, whenever $\alpha\not\in{\sf WO}$, $D=\neg E$ and the Wadge rank of $D$ is greater than or equal to $\psi(|\alpha|)$.

Player I does not have a winning strategy $\tau$.
Otherwise, by $\Sigma^1_1$-bounding, there is an upper bound $\xi$ of ordinals in the image of $\tau$.
Then, $(\psi(\alpha))_{\alpha<\xi}$ gives countably many $\Delta$ sets, and by the closure property of $\Delta$, one can easily obtain a $\Delta$ set whose Wadge rank is greater than or equal to $\sup_{\alpha<\xi}\psi(\alpha)$.
Hence, Player II wins.

By the axiom of determinacy ${\sf AD}$, Player II has a winning strategy $\tau$.
Let $G$ be a universal $\Gamma$ set.
Then, define 
\[A=\{(\alpha,x)\in\om^\om:x\in G_{\pi_0\tau(\alpha)}\}.\]

In other words, $A_\alpha=G_{\pi_0\tau(\alpha)}=\neg G_{\pi_1\tau(\alpha)}$.
Hence, $(A_\alpha)_{\alpha\in{\sf WO}}$ is uniformly $\Delta$.
By the closure property, $\bigsqcup_{\alpha\in{\sf WO}}A_\alpha\in\Delta$, whose Wadge rank is greater than or equal to $(\psi(\alpha))_{\alpha<\om_1}$.
Hence, $\psi$ cannot be a cofinal sequence.
\end{proof}

Under ${\sf AD}$, it is known that ${\sf cf}(\om_n)=\om_2$ whenever $2\leq n<\om$; see \cite[Corollary 28.8]{KanBook}.

\subsection{Upper bound}
By Proposition \ref{prop:omf-closed-coproduct}, $D^\ast_\om(\tpbf{\Pi}^1_1)$ is the minimal Wadge pointclass which is strictly closed under $\om_1$-pwo coproduct.
Therefore, for any $A\in\Delta(D^\ast_\om(\tpbf{\Pi}^1_1))$, the pointclass $\Gamma_A=\{B\subseteq\om^\om:B\leq_{\sf W}A\}$ is not strictly closed under $\om_1$-pwo coproduct.
In this section, we analyze the Wadge rank of such a pointclass. 

Let $(A_\alpha)_{\alpha\in{\sf WO}}$ be a uniform $\Delta$ collection.
Then, for any $\xi<\om_1$, put $A_{<\xi}:=\join_{|\alpha|<\xi}A_\alpha$, where $|\alpha|$ is the order type of $\alpha$ if $\alpha$ is well-ordered.
We say that $\Gamma$ is {\em strictly closed under $(<\om_1)$-coproduct} if, for any uniformly $\Delta$ collection $(A_\alpha)_{\alpha\in{\sf WO}}$ and any $\xi<\om_1$, we have $A_{<\xi}\in\Delta$.

\begin{lemma}\label{lem:cofinal-omega-one}
Assume that $\Gamma$ is strictly closed under $(<\om_1)$-coproduct, but not strictly closed under $\om_1$-pwo coproduct, witnessed by $(A_\alpha)_{\alpha\in{\sf WO}}$.
Then, $(A_{<\xi})_{\xi<\om_1}$ is a cofinal sequence in the Borel-Wadge degrees of $\Delta$ sets.
\end{lemma}

\begin{proof}
Put $A=\join_{\alpha\in{\sf WO}}A_\alpha$.
Then, $A\not\in\Delta$ by our assumption.
If $B\in\Delta$, by Wadge's lemma, we have $B\leq_{\sf W}A$ via some $\theta$ and $B\leq_{\sf W}\neg A$ via some $\eta$.
Let $|x|_\theta$ be the rank of the $1$st corrdinate of $\theta(x)$.
In other words, $|x|_\theta=\alpha$ if and only if $\theta(x)\in A_\alpha$.
Define $|x|_\eta$ in the similar manner.
Then, since $(\neg{\sf WO})\times\om^\om\subseteq\neg A$ and $A\subseteq {\sf WO}\times\om^\om$, we have
\[|x|_\eta=\infty\implies x\in B\implies|x|_\theta<\infty.\]
%
%

Thus, there exists no $x$ such that both ``$|x|_\theta=\infty$'' and ``$|x|_\eta=\infty$'' hold.
Moreover, these properties are $\tpbf{\Sigma}^1_1$.
Hence, the properties ``$|x|_\theta=\infty$'' and ``$|x|_\eta=\infty$'' determine a disjoint pair of $\tpbf{\Sigma}^1_1$ sets.
Therefore, by Lusin's separation theorem, there exists a Borel set $C$ such that
\[|x|_\eta=\infty\implies x\in C\implies |x|_\theta<\infty.\]


In particular, $x\in C$ implies $|x|_\theta<\infty$ and $x\not\in C$ implies $|x|_\eta<\infty$.
%
%
%
%
Since $C$ is Borel, and $\theta$ and $\eta$ are continuous, by $\tpbf{\Sigma}^1_1$-boundedness, there exists $\xi<\om_1$ such that, for any $x\in\om^\om$, $x\in C$ implies $|x|_\theta<\xi$ (i.e., $\theta(x)\in A_{<\xi}$), and $x\not\in C$ implies $|x|_\eta<\xi$ (i.e., $\eta(x)\in A_{<\xi}$).
Now, we define a Borel reduction $\gamma$ as follows:
\[
\gamma(x)=
\begin{cases}
(0,\theta(x))&\mbox{ if }x\in C,\\
(1,\eta(x))&\mbox{ if }x\not\in C.
\end{cases}
\]

Then, we claim that $B$ is Borel-Wadge reducible to $A_{<\xi}\sqcup \neg A_{<\xi}$ via $\gamma$, where $A_{<\xi}\sqcup \neg A_{<\xi}=(\{0\}\times A_{<\xi})\cup(\{1\}\times \neg A_{<\xi})$.
Since $\theta$ witnesses $B\leq_{\sf W}A$, $x\in B$ if and only if $\theta(x)\in A$.
Hence, if $x\in C$ then $x\in B$ if and only if $\gamma(x)=(0,\theta(x))\in\{0\}\times A$, and the latter is equivalent to $\gamma(x)\in\{0\}\times A_{<\xi}$ as we must have $|x|_\theta<\xi$.
Similarly, since $\eta$ witnesses $B\leq_{\sf W}\neg A$, $x\in B$ if and only if $\eta(x)\not\in A$.
Hence, if $x\not\in C$ then $x\in B$ if and only if $\gamma(x)=(1,\eta(x))\in\{1\}\times \neg A$, and the latter is equivalent to $\gamma(x)\in\{1\}\times \neg A_{<\xi}$ as we must have $|x|_\eta<\xi$.
This verifies the claim.
\end{proof}

By combining Lemma \ref{lem:cofinal-omega-one} and Proposition \ref{prop:omf-closed-coproduct}, the desired upper bound can be almost obtained: The Borel-Wadge rank of $D_\om^\ast(\tpbf{\Pi}^1_1)$ is at most $\om_2$.

\subsection{Inside Borel-Wadge degrees}

Unfortunately, Lemma \ref{lem:cofinal-omega-one} only gives a result on Borel-Wadge degrees.
To prove Theorem \ref{thm:main-theorem}, this result has to be transformed into a result for Wadge degrees.

\begin{prop}[{\sf AD}]\label{lem:inside-Borel-Wadge-degrees}
The Wadge rank of $A$ is $\om_2$ if and only if its Borel-Wadge rank is $\om_2$.
\end{prop}

\begin{proof}
Clearly, the Wadge rank of $A$ is greater than or equal to its Borel-Wadge rank.
For the other direction, we claim that if the Wadge rank of $A$ has the cofinality at least $\om_2$, so is its Borel-Wadge rank.
This claim implies that if the Wadge rank of $A$ is $\om_2$ then its Borel-Wadge rank has to be at least $\om_2$, so it concludes the proof.

Assume that the cofinality of the Borel-Wadge rank of $A$ is at most $\om_1$.
Then, there exists a sequence $(A_\xi)_{\xi<\om_1}$ such that $A_\xi<_{\sf BW}A$ for any $\xi<\om_1$, and for any $B<_{\sf BW}A$ we have $B\leq_{\sf BW}A_\xi$ for some $\xi<\om_1$.
Now, fix a total $\tpbf{\Sigma}^0_{\alpha+1}$-measurable function $\lambda_\alpha\colon\om^\om\to\om^\om$ such that for any $\tpbf{\Sigma}^0_\alpha$-measurable function $\theta\colon\om^\om\to\om^\om$ we have $\theta=\lambda_\alpha\circ\eta$ for some continuous function $\eta\colon\om^\om\to\om^\om$.
One can easily construct such a $\lambda_\alpha$; for instance, if $G\subseteq\om^\om\times\om^2$ is a universal $\tpbf{\Sigma}^0_\alpha$ set, then define $\lambda_\alpha(z)(n)=m$ if $m$ is the least number such that $(z,n,m)\in G$; if such an $m$ does not exist, put $\lambda_\alpha(z)(n)=0$.
Note that $B\leq_{\sf BW}C$ if and only if there exists $\alpha<\om_1$ such that $B=\theta^{-1}[C]$ for some $\tpbf{\Sigma}^0_\alpha$-measurable function $\theta$.
The last condition is equivalent to that $B=\eta^{-1}[\lambda_\alpha^{-1}[C]]$ for some continuous function $\eta$.
This means that $B\leq_{\sf W}\lambda_\alpha^{-1}[C]$.
Hence, $B\leq_{\sf BW}C$ if and only if $B\leq_{\sf W}\lambda_\alpha^{-1}[C]$ for some $\alpha<\om_1$.

Put $A^\alpha_\xi=\lambda_\alpha^{-1}[A_\xi]$, and consider the sequence $(A^\alpha_\xi)_{\xi,\alpha<\om_1}$.
Note that we have $A^\alpha_\xi\leq_{\sf W}A$; otherwise, $\neg A\leq_{\sf W}A^\alpha_\xi$ since $\leq_{\sf W}$ is semi-well-ordered under {\sf AD}, and this implies $\neg A\leq_{\sf BW}A_\xi$ by the above characterization of Borel-Wadge reducibility.
Then, however, we have $\neg A\leq_{\sf BW}A_\xi<_{\sf BW}A$, which is impossible (as $\neg A\leq_{\sf BW}A$ implies $\neg A\equiv_{\sf BW}A$).
Hence, $A_\xi^\alpha\leq_{\sf W}A$ for any $\xi,\alpha<\om_1$.
Indeed, $A_\xi^\alpha<_{\sf W}A$ since $A\not\leq_{\sf BW}A_\xi$.
As $(A_\xi)_{\xi<\om_1}$ is cofinal below the Borel-Wadge degree of $A$, for any $B<_{\sf W}A$ there is $\xi<\om_1$ such that $B\leq_{\sf BW}A_\xi$, which means that $B\leq_{\sf W}A_\xi^\alpha$ for some $\alpha<\om_1$.
Hence, $(A^\alpha_\xi)_{\xi,\alpha<\om_1}$ is cofinal below the Wadge degree of $A$.
Consequently, the cofinality of the Wadge rank of $A$ is at most $\om_1$.
\end{proof}

For a set $A\subseteq\om^\om$, recall that the pointclass $\Gamma_A$ is defined as $\{B\subseteq\om^\om:B\leq_{\sf W}A\}$.

\begin{lemma}[{\sf AD}]\label{lem:closed-under-Borel-coproduct}
If the Wadge rank of $A$ is $\om_2$, then $\Gamma_A$ is strictly closed under $(<\om_1)$-coproduct.
\end{lemma}

\begin{proof}
By Proposition \ref{lem:inside-Borel-Wadge-degrees}, if $|A|_{\sf W}=\om_2$ then $|A|_{\sf BW}=\om_2$.
In particular, $|A|_{\sf BW}$ has an uncountable cofinality.
Therefore, by Andretta-Martin \cite[Corollary 17 (a)]{AnMa03}, $A$ is Borel non-self-dual, i.e., $[A]_{\sf BW}\not=[\neg A]_{\sf BW}$.
Then, by \cite[Proposition 20]{AnMa03}, we have $[A]_{\sf W}=[A]_{\sf BW}$.

Let $(A_\alpha)_{\alpha\in{\sf WO}}$ be a uniformly $\Delta_A$ collection, where $\Delta_A=\Gamma_A\cap\neg\Gamma_A$.
Then, there exist $B,C\in\Gamma_A$ such that, whenever $\alpha\in{\sf WO}$, $x\in A_\alpha$ iff $(\alpha,x)\in B$ iff $(\alpha,x)\not\in C$.
We claim that, for any $\xi<\om_1$, $A_{<\xi}$ is Borel-Wadge reducible to $B$ and $\neg C$.
To see this, first note that ${\sf WO}_{<\xi}=\{\alpha\in{\sf WO}:|\alpha|<\xi\}$ is Borel for any $\xi<\om_1$.
Then, consider the reduction $\theta_B$ defined by $\theta_B(\alpha,x)=(\alpha,x)$ if $x\in{\sf WO}_{<\xi}$, and $\theta_B(\alpha,x)=z$ if $x\not\in{\sf WO}_{<\xi}$, where $z$ is an arbitrary element of $\om^\om$ which is not contained in $B$.
Then, $\theta_B$ witnesses that $A_{<\xi}\leq_{\sf BW}B$.
Similarly, one can construct a reduction $\theta_C$ witnessing $A_{<\xi}\leq_{\sf BW}\neg C$.
Since $B,C\leq_{\sf W}A$, we have $A_{<\xi}\leq_{\sf BW}A$ and $A_{<\xi}\leq_{\sf BW}\neg A$.

As discussed above, we have $[A]_{\sf W}=[A]_{\sf BW}\not=[\neg A]_{\sf BW}=[\neg A]_{\sf W}$.
Combining all of these, we obtain that $A_{<\xi}\leq_{\sf W}A$ and $A_{<\xi}\leq_{\sf W}\neg A$.
Therefore, $A_{<\xi}\in\Delta_A$.
This means that $\Gamma_A$ is strictly closed under $(<\om_1)$-coproduct.
\end{proof}

Indeed, the above proof shows that if the Borel Wadge rank of $A$ has an uncountable cofinality, then $\Gamma_A$ is strictly closed under $(<\om_1)$-coproduct.
Now, we give an alternative proof of the Kechris-Martin theorem saying that the Wadge rank of $D_\om^\ast(\tpbf{\Pi}^1_1)$ is $\om_2$.

\begin{proof}[Proof of Theorem \ref{thm:main-theorem}]
By Proposition \ref{prop:omf-closed-coproduct}, $D_\om^\ast(\tpbf{\Pi}^1_1)$ is strictly closed under $\om_1$-pwo coproduct.
Then, by Lemma \ref{lem:steel-fournier-game}, the order type of the Wadge degrees of $\Delta(D_\om^\ast(\tpbf{\Pi}^1_1))$ sets is at least $\om_2$.
If it is greater than $\om_2$, then there exists a $\Delta(D_\om^\ast(\tpbf{\Pi}^1_1))$ set $A\subseteq\om^\om$ whose Wadge rank is exactly $\om_2$.
By Proposition \ref{lem:inside-Borel-Wadge-degrees}, the Borel Wadge rank of $A$ is also $\om_2$.
The minimality of $D_\om^\ast(\tpbf{\Pi}^1_1)$ ensured by Proposition \ref{prop:omf-closed-coproduct} implies that $\Gamma_A$ is not strictly closed under $\om_1$-pwo coproduct.
Moreover, by Lemma \ref{lem:closed-under-Borel-coproduct}, $\Gamma_A$ is strictly closed under $(<\om_1)$-coproduct.
Therefore, by Lemma \ref{lem:cofinal-omega-one}, there exists a cofinal sequence $(A_{<\xi})_{\xi<\om_1}$ of length at most $\om_1$ in the Borel-Wadge degrees of $\Delta_A$ sets.
This implies that the cofinality of $|A|_{\sf BW}$ is at most $\om_1$.
However, since $|A|_{\sf BW}=\om_2$, it contradicts the fact that ${\sf cf}(\om_2)=\om_2$.
\end{proof}

\section{Beyond $\om_2$}

\subsection{$\Pi^1_1$-process with infinite mind-changes}

The relationships among key pointclasses mentioned in Sections \ref{sec:3} and \ref{sec:4} are summarized as in Figure \ref{fig:key-principles}.

\begin{figure}[t]
\includegraphics[width=150mm]{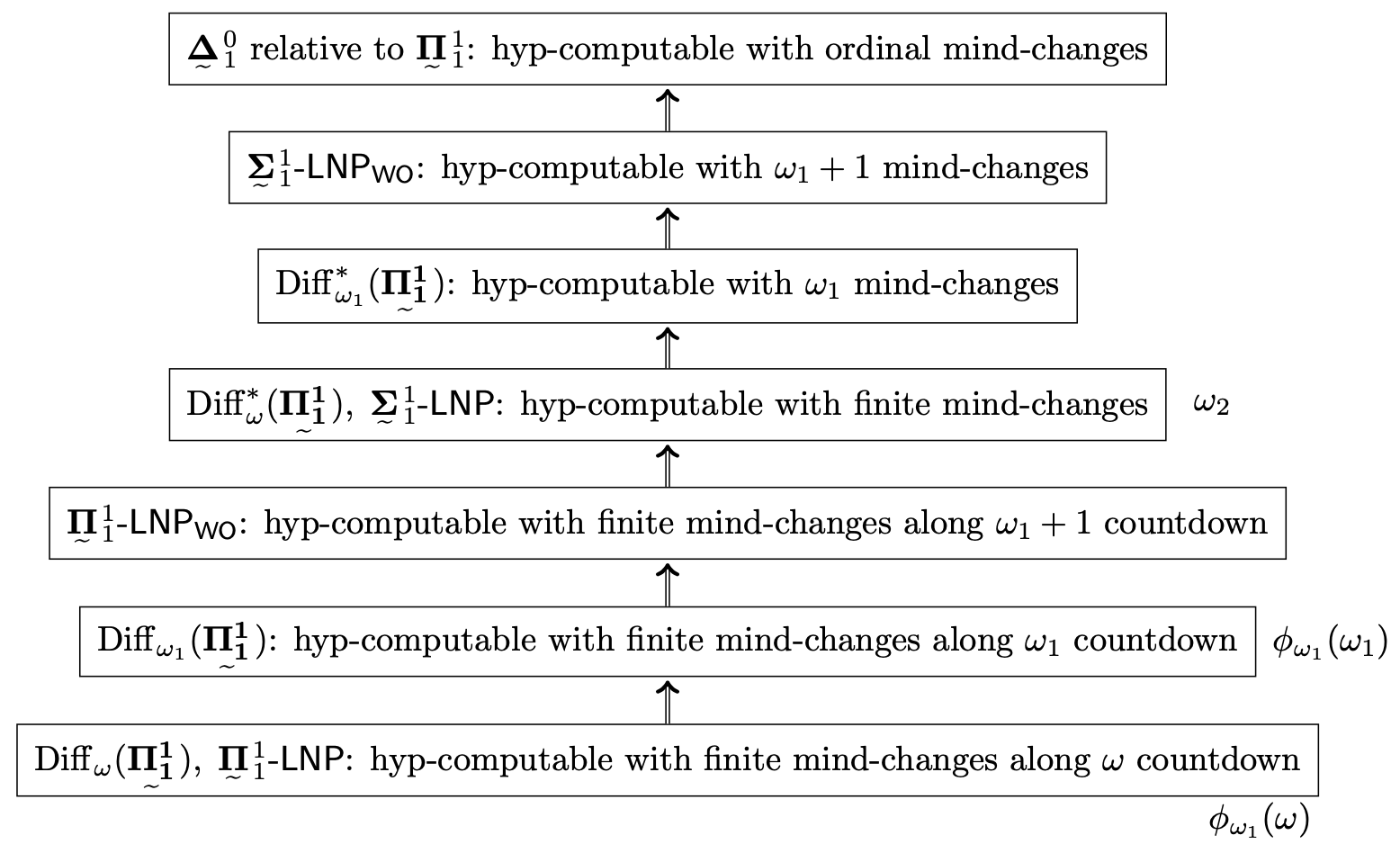}
\caption{Key principles}\label{fig:key-principles}
\end{figure}

We now move to the $(\om+1)$-st level, $(D_{\om+1}^\ast(\tpbf{\Pi}^1_1))$, of the decreasing difference hierarchy.
That is, we consider the following $\om+1$ sequence $(P_\alpha)_{\alpha<\om+1}$ of $\tpbf{\Pi}^1_1$ sets:
\[P_0\supseteq P_1\supseteq P_2\supseteq\dots\supseteq\bigcap_{n<\om}P_n\supseteq P_\om.\]

In this section, we deal with the following question:

\begin{question}
Calculate the Wadge rank of $\Delta(D^\ast_{\om+1}(\tpbf{\Pi}^1_1))$.
\end{question}

To tackle this problem, we first show, perhaps somewhat surprisingly, that any infinite level of the decreasing difference hierarchy is strictly closed under $\om_1$-pwo coproduct even if it is a successor level.

\begin{prop}\label{prop:inf-dd-pwo}
For any infinite ordinal $\eta\geq\om$, $D^\ast_\eta(\tpbf{\Pi}^1_1)$ is strictly closed under $\om_1$-pwo coproduct.
\end{prop}

\begin{proof}
Abbreviate $\Delta(D^\ast_\eta(\tpbf{\Pi}^1_1))$ as $\Delta$.
Let $(A_\alpha)_{\alpha\in{\sf WO}}$ be a uniform $\Delta$ collection.
Then, there exists a sequence $(P_\xi,\check{P}_\xi)_{\xi<\eta}$ of $\tpbf{\Pi}^1_1$ sets such that $A_\alpha=\diff^\ast_{\xi<\eta} P_\xi^{[\alpha]}=\neg\diff^\ast_{\xi_\eta}\check{P}_\xi^{[\alpha]}$ for any $\alpha\in{\sf WO}$, where $S^{[\alpha]}$ is the $\alpha$th section of $S$.
Then put $Q_\xi=\pi_0^{-1}[{\sf WO}]\cap P_\xi$.
Moreover, put $\check{Q}_0=\om^\om$, $\check{Q}_1=\pi_0^{-1}[{\sf WO}]$, and $\check{Q}_{2+\xi}=\pi_0^{-1}[{\sf WO}]\cap \check{P}_\xi$.
Note that $\xi<\eta$ implies $2+\xi<\eta$ since $\eta$ is infinite.
Moreover, $Q_\xi$ and $\check{Q}_\xi$ are $\tpbf{\Pi}^1_1$.

We claim that $\bigsqcup_{\alpha\in{\sf WO}}A_\alpha=\diff^\ast_{\xi<\eta}Q_\xi=\neg\diff^\ast_{\xi<\eta}\check{Q}_\xi$.
Given $(\alpha,x)$, if $\alpha\in{\sf WO}$ then one can easily see that $\max\{\xi:(\alpha,x)\in Q_\xi\}=\max\{\xi:x\in P^{[\alpha]}_\xi\}$ if it exists.
Therefore, $(\alpha,x)\in\diff^\ast_{\xi<\eta}Q_\xi$ if and only if $x\in \diff^\ast_{\xi<\eta}P^{[\alpha]}_\xi=A_\alpha$.
Moreover, if $\alpha\not\in{\sf WO}$ then $(\alpha,x)\not\in Q_0$, and therefore $(\alpha,x)\not\in\diff^\ast_{\xi<\eta}Q_\xi$.
Hence, $\bigsqcup_{\alpha\in{\sf WO}}A_\alpha=\diff^\ast_{\xi<\eta}Q_\xi$.
Again, given $(\alpha,x)$, if $\alpha\in{\sf WO}$ and $\{\xi:x\in \check{P}^{[\alpha]}_\xi\}=\emptyset$, then $x\not\in\diff^\ast_{\xi<\eta}\check{P}^{[\alpha]}_\xi=\neg A_\alpha$, and moreover $(\alpha,x)\in \check{Q}_1\setminus\check{Q}_2$; hence $(\alpha,x)\not\in\diff^\ast_{\xi<\eta}\check{Q}_\xi$.
If $\alpha\in{\sf WO}$ and $\{\xi:x\in \check{P}^{[\alpha]}_\xi\}\not=\emptyset$ then one can easily see that $\max\{\xi:(\alpha,x)\in \check{Q}_\xi\}=2+\max\{\xi:x\in \check{P}^{[\alpha]}_\xi\}$ if it exists.
In particular, both values have the same parity, and therefore, $(\alpha,x)\in\diff^\ast_{\xi<\eta}\check{Q}_\xi$ if and only if $x\in \diff^\ast_{\xi<\eta}\check{P}^{[\alpha]}_\xi=\neg A_\alpha$.
If $\alpha\not\in{\sf WO}$ then $(\alpha,x)\in\check{Q}_0\setminus\check{Q}_1$, and therefore $(\alpha,x)\in\diff^\ast_{\xi<\eta}\check{Q}_\xi$.
Hence, $\bigsqcup_{\alpha\in{\sf WO}}A_\alpha=\neg\diff^\ast_{\xi<\eta}\check{Q}_\xi$.
\end{proof}

As a consequence of Proposition \ref{prop:inf-dd-pwo}, combined with Lemma \ref{lem:steel-fournier-game}, one can see that the Wadge rank of $\Delta(D^\ast_{\om+\eta}(\tpbf{\Pi}^1_1))$ is at least $\om_2\cdot(1+\eta)$ for each $\eta<\om_1$.
As a special case, we conclude that the Wadge rank of $\Delta(D^\ast_{\om+1}(\tpbf{\Pi}^1_1))$ is at least $\om_2\cdot 2$.
In fact, however, one can observe that the Wadge rank of $\Delta(D^\ast_{\om+1}(\tpbf{\Pi}^1_1))$ is not such a small value.
For instance, one can obtain the following lower bound:

\begin{theorem}\label{thm:Wadge-rank-om1}
The Wadge rank of $\Delta(D^\ast_{\om+1}(\tpbf{\Pi}^1_1))$ is greater than $\om_2\cdot\om_1$.
\end{theorem}

We will now prepare a proof of this theorem.
Let $(P_\alpha)_{\alpha<\om+1}$ be a decreasing $\om+1$ sequence of $\tpbf{\Pi}^1_1$ sets.
If moreover we have a $\tpbf{\Pi}^1_1$ set $\check{P}_\om$ such that $\bigcap_{n<\om}P_n=P_\om\cup\check{P}_\om$ and  $P_\om\cap\check{P}_\om=\emptyset$, we call the sequence $(P_n,\check{P}_\om)_{n\leq\om}$ {\em type $\Delta(\om+1)$.}
A decreasing $\om+1$ sequence $(P_\alpha)_{\alpha<\om+1}$ defines a set $P$ as in the usual difference hierarchy; that is, at the first $\om$ levels, a hyp-computable learner proceeds as follows:
\[0\to 1\to 0\to \dots\]

If the guess changes infinitely many often, then the guess becomes $0$.
After that, we will be able to change the guess to $1$:
\[0\to 1\to 0\to \cdots(\om\mbox{ changes})\cdots 0\to 1\]

A type $\Delta(\om+1)$ sequence $(P_\alpha,\check{P}_\om)_{\alpha<\om+1}$ defines a set $P$ in a similar manner, where if the guess changes infinitely many often (which means $x\in\bigcap_{n<\om}P_n$) then we soon decide the final value is $0$ or $1$ (which corresponds to either $x\in\check{P}_\om$ or $x\in{P}_\om$):
\[0\to 1\to 0\to \cdots(\om\mbox{ changes})\cdots i\]

\begin{lemma}\label{lemma:om-plus-one}
A set $A$ is defined by a type $\Delta(\om+1)$ sequence if and only if $A\in\Delta(D^\ast_{\om+1}(\tpbf{\Pi}^1_1))$.
\end{lemma}

\begin{proof}
The forward direction is trivial.
For the backward direction, we have two sequences $P,Q$ of type $\om+1$ guessing $A$.
Given $x$, the first guessing process $P$ returns $0$ when the guess changes infinitely often but does not declare the $\om$th mind-change, i.e., $x\in\bigcap_n P_n$ but $x\not\in P_\om$.
Another guessing process $Q$ returns $1$ when $x\in\bigcap_n Q_n$ but $x\not\in Q_\om$.
We construct a guessing sequence $D$ of type $\Delta(\om+1)$.

At stage $\alpha$, if the number of changes of $P$ is smaller than $Q$, then the process $D$ emulates $P$; otherwise, $D$ emulates $Q$ whenever at least one of the numbers is finite.
In other words, compare $n_P=\sup\{n<\om:x\in P_n[\alpha]\}$ and $n_Q=\sup\{n<\om:x\in Q_n[\alpha]\}$.
If $t=\min\{n_P,n_Q\}$ is finite and even, returns $1$.
If $t=\min\{n_P,n_Q\}$ is finite and odd, returns $0$.
If $t$ is infinite, we have $x\in\bigcap_nP_n[\alpha]$ and $x\in\bigcap_nQ_n[\alpha]$.
In this case, either $P$ or $Q$ declare the $\om$th mind-change; otherwise, $P$'s final guess is ``$A(x)=0$'' but $Q$'s final guess is ``$A(x)=1$'', which is impossible.
Thus, wait for seeing stage $\beta\geq\alpha$ such that either $P$ or $Q$ declare the $\om$th mind-change, i.e., $x\in P_\om[\beta]$ or $x\in Q_\om[\beta]$.
In the former case, $D$'s final guess is ``$A(x)=1$'', i.e., $x\in D_\om$.
In the latter case, $D$'s final guess is ``$A(x)=0$'', i.e., $x\in \check{D}_\om$.
It is not hard to check that $D$ gives a process of type $\Delta(\om+1)$ guessing $A$.
\end{proof}

\subsection{$\om$-change matrix}

In order to prove Theorem \ref{thm:Wadge-rank-om1}, it suffices to show that there are at least $\om_1$ many classes between $D^\ast_{\om}(\tpbf{\Pi}^1_1)$ and $D^\ast_{\om+1}(\tpbf{\Pi}^1_1)$.
First, we observe that there are at least $\om$ many such classes.
A key observation is that, as we have seen above, $\Delta(D^\ast_{\om+1}(\tpbf{\Pi}^1_1))$ corresponds to hyp-computability with at most $\om$ mind-changes.
What we will show is that there is a finer hierarchy within hyp-computability with at most $\om$ mind-changes.
The following definition is hard to understand, so we give an intuitive explanation after the definition.

\begin{definition}\label{def:om-change}
A double sequence $A=(A^j_n)_{(j,n)\in \ell\times\om}$ of $\tpbf{\Pi}^1_1$ sets is called an {\em $\om$-change $\ell\times\om$ matrix} if the following holds (w.r.t.~some approximation of $(A^j_n)_{(j,n)\in \ell\times\om}$):
\begin{enumerate}
\item For any $j<\ell$, $(A^j_n)_{n\in\om}$ is a decreasing sequence.
\item For any $j<k<\ell$ and $\alpha\in{\sf WO}$, we have $A^k_0[\alpha]\cap A^j_n\subseteq A^j_n[\alpha]$ for any $n\in\om$.
\end{enumerate}

Given $c$, we define a new difference operator $c\diff'_{\ell\times\om}$, which takes an $\om$-change $\ell\times\om$ matrix $A$ and an $\ell\times\om$ matrix $a=(a^j_n)$ as input.
To define this operator, we first introduce auxiliary parameters $v_k(x)$ for each $k\leq\ell$.
Then, we first put $v_0(x)=c$.
For each $k<\ell$, define $v_{k+1}$ as follows:
\[
v_{k+1}(x)=
\begin{cases}
a_m^k&\mbox{ if }m=\max\{n<\om:x\in A_n^k\},\\
v_k(x)&\mbox{ if no such $m$ exists.}
\end{cases}
\]

Then we define $c\diff'_{\ell\times\om}[a/A]$ as follows:
\[
c\diff_{\ell\times\om}{}\!\!^\prime[a/A](x)=
\begin{cases}
v_k(x)&\mbox{ if }k=\min\left\{j<\ell:x\in\bigcap_{n<\om}A^j_n\right\},\\
v_\ell(x)&\mbox{ if no such $k$ exists.}
\end{cases}
\]

%
%
%
Let $cD'_{\ell\times\om}(\tpbf{\Pi}^1_1)$ be the class of all sets of the form $c\diff'_{\ell\times\om}[a/A]$ for some $\om$-change $\ell\times\om$ matrix $A=(A^j_n)$ of $\tpbf{\Pi}^1_1$ sets and $\ell\times\om$ matrix $a=(a^j_n)$ with $a_n^j\in\{0,1\}$.
If $c=0$ and $a_n^j={\sf par}(n)$ we simply write $D'_{\ell\times\om}(\tpbf{\Pi}^1_1)$.
\end{definition}

Let us explain an intuitive meaning of this definition.
Each row of an $\ell\times\om$ matrix acts in the same way as the class $vD^\ast_\om(\tpbf{\Pi}^1_1)$ for some $v$.
In other words, for each $k<\ell$, a hyp-computable process $\Psi_k$ is assigned to the $k$-th row, which may change the guess at most $\om$ times, and when the $\om$-th change occurs, the final guess is set to $v_k$.
The first guess is also set to $v_k$.
\begin{align*}
A^k_0\supseteq A^k_1\supseteq A^k_2\supseteq\cdots(\om\mbox{ changes})&\cdots\\
v_k\to a^k_0\to a^k_1\to a^k_2\to \cdots(\om\mbox{ changes})&\cdots v_k
\end{align*}

However, this value $v_k$ can vary if $k>0$.
A candidate $a^{k-1}_s$ for the value $v_k$ is determined by a guessing process $\Psi_{k-1}$ assigned to the row just one above it.
However, $\Psi_{k-1}$ may also change the guess $\om$ times, so the final value $v_{k-1}$ depends on a guessing process $\Psi_{k-2}$ assigned to the $(k-1)$-th row if $k-1>0$.
Continue this argument, and if this process arrives the $0$-th row, and if the $\om$-th change of $\Psi_0$ occurs,  then the final guess is set to $c$.
Note, however, that although this explanation seems to proceed in order from the bottom row, the condition (2) in Definition \ref{def:om-change}  guarantees that the process starts from the top row; that is, if we start the guessing process in some row, then the guessing processes in the rows above it has already terminated.
This assumption ensures that the guesses in each line can be integrated into a single hyp-computable process with at most $\om$ mind-changes:

\begin{lemma}\label{lem:om-change-1}
For any $\ell<\om$, $D'_{\ell\times\om}(\tpbf{\Pi}^1_1)\subseteq \Delta(D^\ast_{\om+1}(\tpbf{\Pi}^1_1))$.
\end{lemma}

\begin{proof}
Let $P\in D'_{\ell\times\om}(\tpbf{\Pi}^1_1)$ be given.
Then, $P$ is of the form $c\diff'_{\ell\times\om}[a/A]$ for some $\om$-change $\ell\times\om$ matrix $A=(A^j_n)$ and $\ell\times\om$ matrix $a=(a^j_n)$, where $c=0$ and $a^j_n={\sf par}(n)$.
To simplify our argument, one can assume that for any $x$ and $\alpha$ there are at most one $(j,n)$ such that $x$ is enumerated into $A^j_n$ at stage $\alpha$; that is, $x\in A^j_n[\alpha]$ but $x\not\in A^j_n[\beta]$.
For instance, one can assume that we only enumerate something into $A^j_n$ at stage $\om^2\cdot\alpha+\om\cdot j+n$ for some $\alpha$.
We construct a guessing sequence $D$ of type $\Delta(\om+1)$.

Our guessing algorithm to compute $P(x)$ can be described as follows:
\begin{itemize}
\item At each stage $\alpha$, starting from the top row, one can calculate the current value $v_j(x)[\alpha]$ of $v_j$ for each $j\leq\ell$.
Indeed, it is easy to check that $(\alpha,k,x)\mapsto v_k(x)[\alpha]$ is Borel.
\item As the first case, if mind-changes occur infinitely often at some row $k$, then the guess is set to $v_k(x)[\alpha]$, and the computation terminates.
The condition (2) in Definition \ref{def:om-change} guarantees that at most one row is active at any stage $\alpha$, and thus, there is at most one row $j$ at which mind-changes occur infinitely often at $\alpha$.
Moreover, the condition (2) inductively ensures that the value of $v_j(x)$ will not change after stage $\alpha$ for any $j\leq k$, so the guess $v_k(x)[\alpha]$ matches the output value $P(x)$.
\item As the second case, if the mind-changes has not yet occurred infinitely many times at any row, then the algorithm currently guesses that the output value of $P(x)$ is $v_{\ell}(x)[\alpha]$.
Since $v_\ell(x)$ only changes when mind-changes occur at some row, and there are only a finite number of rows, the number of times of mind-changes has is kept finite in this case.
\end{itemize}

To be more precise, first check if there exists $k$ such that $x\in\bigcap_{n<\om} A^k_m[\alpha]$.
If true, this is the first case.
If $\alpha$ is the least such stage, and $k$ is the least such row, then our algorithm returns $v_k(x)$. 
By the condition (2) in Definition \ref{def:om-change}, since $x\in A^k_0[\alpha_0]$ we have $A^j_n[\alpha]=A^j_n$ for any $j<n$.
This means that $v_k(x)[\alpha]=v_k(x)$, and $k$ is the least row such that $x\in\bigcap_{n<\om} A^k_m$.
Hence $c\diff'_{k<\om}[a/A](x)=v_k(x)$ by definition.
If there exists no $k$ such that $x\in\bigcap_{n<\om} A^k_m[\alpha]$, then this is the second case.
If this is true for any stage $\alpha$, then $x\not\in\bigcap_{n<\om} A^k_m$ for any $k<\ell$, so the output of our guessing algorithm converges to $v_\ell(x)$.
In this case, by the definition of $\diff'$, we also have $c\diff'_{k<\om}[a/A](x)=v_\ell(x)$.
Hence, our algorithm correctly guess the value of $P(x)=c\diff'_{k<\om}[a/A](x)$.

As mentioned above, the mind-changes in the guess of our algorithm due to the second case occur only a finite number of times, and once the first case is reached, the guess never changes.
Also, in the first case, the guess is determined immediately.
Hence, this is a $\Delta(\om+1)$ guessing process.
This completes the proof.
\end{proof}

To ensure that it is a reasonable pointclass, it should be closed under continuous substitution.

\begin{lemma}\label{lem:om-change-2}
$D'_{\ell\times\om}(\tpbf{\Pi}^1_1)$ is closed under continuous substitution, that is, $B\leq_{\sf W}A\in D'_{\ell\times\om}(\tpbf{\Pi}^1_1)$ implies $B\in D'_{\ell\times\om}(\tpbf{\Pi}^1_1)$.
\end{lemma}

\begin{proof}
More generally, let $f\in cD'_{\ell\times\om}(\tpbf{\Pi}^1_1)$ be given, and assume that $g=f\circ\theta$ for some continuous function $\theta$.
It suffices to show that $g\in cD'_{\ell\times\om}(\tpbf{\Pi}^1_1)$.
Then, $f$ is of the form $c\diff'_{\ell\times\om}[a/A]$ for some $\om$-change $\ell\times\om$ matrix $A=(A^j_n)$ and $\ell\times\om$ matrix $a=(a^j_n)$.
Then, put $B^j_n=\theta^{-1}[A^j_n]$, and then $B^j_n[\alpha]=\theta^{-1}[A^j_n[\alpha]]$ yields an approximation of $B^j_n$ for any $j<\ell$ and $n\in\om$.
The property that $A$ is an $\om$-change matrix is inherited by $B=(B^j_n)$.
Moreover, one can see that $g=c\diff'_{\ell\times\om}[a/B]$ since $\theta^{-1}[\bigcap_{n<\om}A^k_n]=\bigcap_{n<\om}\theta^{-1}[A^k_n]$, and $\theta^{-1}[A^k_n\setminus A^k_{n+1}]=\theta^{-1}[A^k_n]\setminus\theta^{-1}[A^k_{n+1}]$.
Therefore, $g\in cD'_{\ell\times\om}(\tpbf{\Pi}^1_1)$.
\end{proof}

\begin{lemma}\label{lem:om-change-3}
For any $\ell<\om$, $D'_{\ell\times\om}(\tpbf{\Pi}^1_1)$ is strictly closed under $\om_1$-pwo coproduct.
\end{lemma}

\begin{proof}
Put $\Delta=\Delta(D'_{\ell\times\om}(\tpbf{\Pi}^1_1))$.
Let a uniform $\Delta$-collection $(P_\alpha)_{\alpha\in{\sf WO}}$ be given.
Then it is obtained by a collection of $\om$-change matrices $A_\alpha=(A^{\alpha,j}_n)$.
Put $B_{n}^{j}=\bigsqcup_{\alpha\in{\sf WO}}A^{\alpha,j}_n$ for each $n<\om$.
Then, $B_n^j$ is $\tpbf{\Pi}^1_1$ for each $n<\om$ since the $\om_1$-pwo-coproduct of a uniform collection of $\tpbf{\Pi}^1_1$ sets is $\tpbf{\Pi}^1_1$.
Moreover, $B=(B^{j}_n)_{j,n}$ is an $\om$-change matrix:
For the condition (2), if $j<k$ and $(\alpha,x)\in B_0^k[\beta]\cap B^j_n$ then $x\in A_0^{\alpha,k}[\beta]\cap A^{\alpha,j}_n\subseteq A^{\alpha,j}_n[\beta]$, so $(x,\alpha)\in B^j_n[\beta]$.

To see the equality $\diff'_{\ell\times\om}B=\bigsqcup_{\alpha\in{\sf WO}}P_\alpha$, let $(\alpha,x)$ be given.
Clearly, if $\alpha\not\in{\sf WO}$ then $(\alpha,x)\not\in B_0^j$, and therefore, $x\not\in\diff'_{\ell\times\om}B$.
If $\alpha\in{\sf WO}$, then $(\alpha,x)\in B_n^j$ if and only if $x\in A_n^{\alpha,j}$.
Therefore, $(\alpha,x)\in\diff'_{\ell\times\om}B$ if and only if $x\in\diff'_{\ell\times\om}A_\alpha=P_\alpha$.
This completes the proof.
\end{proof}

\begin{lemma}\label{lem:om-change-4}
The hierarchy $(D'_{\ell\times\om}(\tpbf{\Pi}^1_1))_{\ell<\om}$ does not collapse; that is, for any $k<\ell<\om$, $D'_{\ell\times\om}(\tpbf{\Pi}^1_1)\setminus D'_{k\times\om}(\tpbf{\Pi}^1_1)$ is nonempty.
\end{lemma}

\begin{proof}
We first construct a universal $D'_{\ell\times\om}(\tpbf{\Pi}^1_1)$ set $G$.
The existence of a universal $\tpbf{\Pi}^1_1$ set clearly yields a total representation of all $\ell\times\om$ matrices of $\tpbf{\Pi}^1_1$ sets which are not necessarily $\om$-change matrices.
Given $\ep\in\om^\om$, let $(A^j_n)$ be the $\ell\times\om$ matrix coded by $\ep$.
Then, define an $\ell\times\om$ matrix $(B^j_n)$ as follows:
Given $\alpha<\om_1$, if $x\not\in B_0^k[\alpha]$ for any $k>j$ then we declare that $x\in B_n^j[\alpha]$ if and only if $x\in A_n^j[\alpha]$.
If $x\in B_0^k[\alpha]$ for some $k>j$, declare that $x\in B_n^j[\alpha]$ if and only if $x\in B_n^j[\beta]$ for some $\beta<\alpha$.
That is, once a mind-change occurs in a lower row $k>j$, no more changes in the row $j$ will occur.
Then, put $B_n^j=\bigcup_{\alpha<\om_1}B_n^j[\alpha]$, and then it is easy to see that $B_\ep=(B_n^j)_{(n,j)\in\ell\times\om}$ is an $\om$-change matrix of $\tpbf{\Pi}^1_1$ sets.
Clearly, for any $\om$-change $\ell\times\om$ matrix $A$ there exists $\ep$ such that $B_\ep=A$.
We define $G$ as the set of all $(\ep,x)$ such that $x\in\diff'_{\ell\times\om} B_\ep$.
Note that $G\in D'_{\ell\times\om}(\tpbf{\Pi}^1_1)$ since $G$ is of the form $\diff'_{\ell\times\om} C$ for the $\om$-change matrix $C=(C^j_n)$ defined by $C^j_n=\{(\ep,x):x\in B^{\ep,j}_n\}$, where $B_\ep=(B^{\ep,j}_n)$.

Next, it is easy to see that the dual class of $D'_{\ell\times\om}(\tpbf{\Pi}^1_1)$, i.e., $\neg D'_{\ell\times\om}(\tpbf{\Pi}^1_1)$, is also included in $\Delta(D'_{(\ell+1)\times\om}(\tpbf{\Pi}^1_1))$, by shifting the components of each row by one, and by adding the topmost row which always guesses $1$.
Hence, it remains to show that $\neg D'_{\ell\times\om}(\tpbf{\Pi}^1_1)$ is not included in $D'_{\ell\times\om}(\tpbf{\Pi}^1_1)$.
The rest of the proof is an easy diagonalization argument.
Let us consider $Q=\{x:(x,x)\not\in G\}$.
We claim that $Q$ does not belong to $D_{\ell\times\om}(\tpbf{\Pi}^1_1)$.
Otherwise, there exists $x$ such that $Q=\diff'_{\ell\times\om}B_x$.
However, $x\in Q$ if and only if $(x,x)\not\in G$ if and only if $x\not\in \diff'_{\ell\times\om}B_x$, a contradiction.
This concludes the proof.
\end{proof}

Lemmas \ref{lem:om-change-1}, \ref{lem:om-change-2}, \ref{lem:om-change-3}, and \ref{lem:om-change-4}, combined with Lemma \ref{lem:steel-fournier-game}, imply that the Wadge rank of $\Delta(D^\ast_{\om+1}(\tpbf{\Pi}^1_1))$ is at least $\om_2\cdot \om$.
It is straightforward to consider the transfinite version of this argument.
That is, for a limit ordinal $\lambda$, one may define $D'_{\lambda\times\om}(\tpbf{\Pi}^1_1)$ as the class of all sets which can be written as countable disjoint unions of sets from $D'_{\lambda[n]\times\om}(\tpbf{\Pi}^1_1)$, $n<\om$, where $(\lambda[n])_{n<\om}$ a fundamental sequence for $\lambda$.
For a successor ordinal $\xi=\zeta+1$, in order to define $D'_{\xi\times\om}(\tpbf{\Pi}^1_1)$, one can simply add one more row to $D'_{\zeta\times\om}(\tpbf{\Pi}^1_1)$.

As a consequence, inside $\Delta(D_{\om+1}^\ast(\tpbf{\Pi}^1_1))$, there are at least $\om_1$ many classes strictly closed under $\om_1$-pwo coproduct.
Hence, by Lemma \ref{lem:steel-fournier-game}, we conclude that the Wadge rank of $\Delta(D_{\om+1}^\ast(\tpbf{\Pi}^1_1))$ is at least $\om_2\cdot\om_1$.
However, by using $\om_1$-pwo coproduct to combine these $\om_1$ many classes, we can create a new class inside $\Delta(D_{\om+1}^\ast(\tpbf{\Pi}^1_1))$.
This concludes the proof of Theorem \ref{thm:Wadge-rank-om1}.
By repeating this process, it seems possible to construct $\om_1+\om_1$ many, $\om_1^2$ many, or $\om_2$ many different classes strictly closed under $\om_1$-pwo coproduct.
If this is the case, by Lemma \ref{lem:steel-fournier-game}, one can show that the Wadge rank of $\Delta(D_{\om+1}^\ast(\tpbf{\Pi}^1_1))$ is at least $\om_2^2$.

\begin{question}
Under ${\sf AD}$, is the Wadge rank of $\Delta(D_{\om+1}^\ast(\tpbf{\Pi}^1_1))$ at least $\om_2^2$?
\end{question}

One may also ask a similar question:

\begin{question}
Under ${\sf AD}$, is the Wadge rank of $\Delta(D_{\om+n}^\ast(\tpbf{\Pi}^1_1))$ at least $\om_2^{n+1}$?
\end{question}


We now move to the next level of $(D_{\om+n}^\ast(\tpbf{\Pi}^1_1))_{n<\om}$.
It is reasonable to ask the following question.

\begin{question}\label{que:rank-forward-backward}
Under {\sf AD}, calculate the Wadge rank of $\Delta(D^\ast_{\om+\om}(\tpbf{\Pi}^1_1))$.
\end{question}

However, we have the impression that answering this question is incredibly difficult.
This is because we feel that there is also a tremendously vast hierarchy between $(D_{\om+n}^\ast(\tpbf{\Pi}^1_1))_{n<\om}$ and $\Delta(D^\ast_{\om+\om}(\tpbf{\Pi}^1_1))$.
The first step is given by ``$\Pi^1_1$-processes with [forward $\om$]$+$[backward $\om$] mind-changes''.
More precisely, we consider the following $\om+\om$ sequence $(P_\alpha)_{\alpha<\om+\om}$ of $\tpbf{\Pi}^1_1$ sets:
\[P_0\supseteq P_1\supseteq P_2\supseteq\dots\supseteq\bigcap_{n<\om}P_n\supseteq\bigcup_{n<\om}P_{\om+n}\supseteq\dots\supseteq P_{\om+2}\supseteq P_{\om+1}\supseteq P_\om.\]

We call such a sequence {\em type $\omf+\omb$}.
If moreover we have $\bigcap_{n<\om}P_n=\bigcup_{n<\om}P_{\om+n}$, we call it {\em type $\Delta(\omf+\omb)$.}
A type $\omf+\omb$ sequence $(P_\alpha)_{\alpha<\om+\om}$ defines a set $P$ as in the usual difference hierarchy; that is, at the first $\om$ levels, a $\Pi^1_1$ guess proceeds as follows:
\[0\to 1\to 0\to \cdots(\om\mbox{ changes})\cdots 0\]

If the guess changes infinitely many often, then the guess becomes $0$.
After that, we will have a fresh mind-change counter $\om$ controlling our next finite mind-changes.

\begin{question}
Under {\sf AD}, calculate the Wadge rank of $\Delta(\omf+\omb)$.
\end{question}

In general, one can consider ``$\Pi^1_1$-processes with [forward $\om$]$+$[backward $\alpha$] mind-changes'' for any $\alpha<\om_1$.
Then we get the corresponding pointclass $\Delta(\omf+\alpha^{\leftarrow})$, and we still have $\Delta(\omf+\alpha^{\leftarrow})\subseteq\Delta(D^\ast_{\om+\om}(\tpbf{\Pi}^1_1))$ for any $\alpha<\om_1$.
Based on these observations, we conjecture that the answer to Question \ref{que:rank-forward-backward} is at least $\om_2^{\om_2}$, but we do not have a method to calculate this at this time.

\begin{ack}
The author is very grateful to Kenta Sasaki for thorough discussions and valuable comments.
The author's research was partially supported by JSPS KAKENHI Grant Number 19K03602, and the JSPS-RFBR Bilateral Joint Research Project JPJSBP120204809.
\end{ack}

\bibliographystyle{plain}
\bibliography{omega}

\end{document}